\documentclass[reqno]{amsart}
\usepackage{amsmath}
\usepackage{xcolor}
\usepackage{amsfonts}
\usepackage{amssymb}
\usepackage{amstext}
\usepackage{amsbsy}
\usepackage{amsopn}
\usepackage{amsthm}
\usepackage{enumerate}
\usepackage{amsxtra}
\usepackage{upref}
\usepackage{graphicx}
\usepackage{epstopdf}
\usepackage{tikz-cd}

\usepackage{hyperref}
\usepackage[colorinlistoftodos]{todonotes}

\DeclareFontFamily{OML}{rsfs}{\skewchar\font'177}
\DeclareFontShape{OML}{rsfs}{m}{n}{ <5> <6> rsfs5 <7> <8> <9> rsfs7
  <10> <10.95> <12> <14.4> <17.28> <20.74> <24.88> rsfs10 }{}
\DeclareMathAlphabet{\mathfs}{OML}{rsfs}{m}{n}

\newtheorem{theorem}{Theorem}[section]
\newtheorem*{theorem*}{Theorem}
\newtheorem{lemma}{Lemma}[section]

\newtheorem*{example*}{Example}
\newtheorem{definition}{Definition}[section]
\newtheorem{corollary}{Corollary}[section]

\numberwithin{equation}{section}

\def\var{\textrm{osc}}

\renewcommand{\epsilon}{\varepsilon}

\def\text#1{\textrm{#1}}

\def\emptyset{\varnothing}

\def\vf{\varphi}

\def \R{\mathbb R}
\def \N{{\mathbb N}}
\def\E{\mathbb E}
\def \Z{\mathbb Z}

\def\ov{\overline}
\def\un{\underline}

\def\C{\mathbb C}

\def\({\biggl(}
\def\){\biggr)}

\def\<{\bold\langle}
\def\>{\bold\rangle}

\def\eps{\varepsilon}

\DeclareMathOperator\const{const}

\DeclareMathOperator\Var{Var}

\def\p{\mathfrak p}
\def\q{\mathfrak q}

\def\eq0{{m_0}}
\def\eqp{{\mu_\phi}}
\def\mm0{{\mu_0}}
\DeclareMathOperator\Log{Log}

\title[EKP inequality]{Effective intrinsic ergodicity for countable state  Markov shifts}
\author{Ren\'e R\"uhr and Omri Sarig}
\thanks{The authors were partially supported by  ISF grant 1149/18 and BSF grant 2016105.}
\date{\today}

\address{Faculty of Mathematics and Computer Science\\ The Weizmann Institute of Science\\ POB 26, Rehovot, Israel}
\email{omsarig@gmail.com reneruehr@gmail.com}

\begin{document}

\begin{abstract}
For strongly positively recurrent countable state Markov shifts, we bound  the distance between an invariant  measure  and the  measure  of maximal entropy in terms of the difference of their entropies. This extends an earlier result for subshifts of finite type, due to Kadyrov. We provide a similar bound for equilibrium measures of strongly positively recurrent potentials, in terms of the pressure difference. For  measures with nearly maximal entropy, we have new, and  sharp, bounds. The strong positive recurrence condition is necessary.
\end{abstract}
\maketitle

\begin{center}
{\em Dedicated to Benjy Weiss on the occasion of his eightieth birthday}
\end{center}

\section{Introduction and summary of main results}
Topological dynamical systems with unique measures of maximal entropy  are called  {\em intrinsically ergodic} \cite{Weiss-Intrisic}.
This property  is weaker than unique ergodicity, and this weakening is  useful, because it allows for many more examples. Natural instrinsically ergodic systems which are not uniquely ergodic can be found in symbolic dynamics \cite{Parry-Intrinsic-MC}, \cite{Gurevich-Measures-Of-Maximal-Entropy}, \cite{Bowen-UES}, \cite{Climenhaga-Thompson-beta}, \cite{Climenhaga-CMP}, \cite{Pavlov-Proc-AMS};
 one-dimensional dynamics \cite{Hofbauer-Intrinsic-Erg}, \cite{Buzzi-Interval-Maps}; the theory of diffeomorphisms \cite{Adler-Weiss-Similarity-Toral-Automorphisms}, \cite{Bowen-UES}, \cite{Buzzi-Crovisier-Sarig-MME}; and in the theory of geodesic flows \cite{Knieper}, \cite{Burns-Climenhaga-Fisher-Thompson}, \cite{Climenhaga-Knieper}.  (This list of references is incomplete, the relevant literature is too plentiful to survey.)

Although weaker than unique ergodicity, intrinsic ergodicity is powerful enough to  have many applications. These include classification problems in ergodic theory  \cite{Adler-Weiss-Similarity-Toral-Automorphisms}; the foundations of statistical mechanics \cite{Ruelle-Z-nu}, \cite{Ruelle-TDF-book}, \cite{Sinai-Gibbs}, the analysis of  periodic orbits \cite{Bowen-Closed-Geodesics}, \cite{Parry-Pollicott-Asterisque}; and  number theory \cite{Einsiedler-Lindenstrauss-Philippe-Venkatesh}. (Again, these are very partial lists.)

  Here we will focus on  the connection between intrinsic ergodicity and equidistribution of measures with high entropy, in the  case of topological Markov shifts.

Consider for example a topologically transitive subshift of finite type $\sigma:\Sigma^+\to\Sigma^+$. This system has a  unique measure of maximal entropy  $\mu_0$ \cite{Parry-Intrinsic-MC}.  Subshifts of finite type are compact, and their  entropy map $\mu\mapsto h_\mu(\sigma)$ is upper semi-continuous in the weak star topology. Together with  intrinsic ergodicity, this easily implies that for every sequence of  shift invariant probability measures $\mu_n$,
\begin{equation}\label{weak-EKP}
\text{ if $h_{\mu_n}(\sigma)\to h_{\mu_0}(\sigma)$, then }\mu_n\to\mu_0\text{ weak star.}
\end{equation}

Kadyrov gave a bound for the rate of convergence \cite{Kadyrov-Effective-Uniqueness}. He showed that there exist constants $C_\beta$ such that for every shift invariant probability measure $\mu$ and for every H\"older continuous function $\psi:\Sigma^+\to\R$ with H\"older exponent $\beta$,
\begin{equation}\label{EKP-ineq}
\left|\int\psi d\mu-\int\psi d\mu_0\right|\leq C_\beta\|\psi\|_\beta\sqrt{h_{\mu_0}(\sigma)-h_\mu(\sigma)},
\end{equation}
where $\|\psi\|_\beta$ is the $\beta$-H\"older norm of $\psi$ (see \S\ref{s.setup}).

Quantitative versions of \eqref{weak-EKP} like \eqref{EKP-ineq} are called  {\em effective intrinsic ergodicity estimates}.
They  first appeared in the doctoral thesis of F.\ Polo \cite{Polo-PhD} for the $\times 2$ map on $\R/\Z$  and for hyperbolic toral automorphisms, but with a cubic root instead of a square root. Polo credits  M.\ Einsiedler for outlining the proof for the $\times 2$ map, and  we will henceforth call \eqref{EKP-ineq} the {\em Einsiedler-Kadyrov-Polo (EKP) inequality. }
For similar inequalities for other systems, see  \cite{Kadyrov-Effective-Uniqueness}, \cite{Ruhr-Effective-Uniqueness}, \cite{khayutin2017large}.

\medskip
In this paper we extend the EKP inequality to topologically transitive countable state Markov shifts. The new features here are non-compactness, the possibility of escape of mass to infinity, and ``phase transitions": non-analytic pressure functions.

\medskip
Of course, the EKP inequality cannot be expected to hold for all countable Markov shifts.
For \eqref{weak-EKP} or \eqref{EKP-ineq} to make sense, we must
at the very least assume that a measure of maximal entropy $\mu_0$ exists, and that $h_{\mu_0}(\sigma)<\infty$.

A deeper observation, due to S.\ Ruette, is that even under these additional assumptions, \eqref{weak-EKP} and \eqref{EKP-ineq} may fail.
It follows from \cite{Ruette} (see also \cite{Gurevich-Savchenko}, \cite{Gurevich-Zargaryan}),
that if \eqref{weak-EKP} or \eqref{EKP-ineq} hold, then
 $\sigma:\Sigma^+\to\Sigma^+$ must be {\em strongly positively recurrent (SPR)}, a condition whose definition is recalled in \S\ref{s.SPR-def}. By the work of Gurevich \cite{Gurevich-Measures-Of-Maximal-Entropy}, topologically transitive SPR countable Markov shifts have a unique measure of maximal entropy.

Thus the right context for studying effective intrinsic ergodicity for countable Markov shifts is the class of topologically transitive SPR shifts.

G.\ Iommi, M.\ Todd and A.\ Velozo have studied the semi-continuity properties of the entropy map for countable Markov shifts  \cite{Iommi-Todd-Velozo-USC}, and proved in \cite[Theorem~8.12]{iommi2019escape} that all SPR topologically transitive countable state Markov shifts satisfy \eqref{weak-EKP}.
Our main result is: {\em All SPR topologically transitive countable state Markov shifts satisfy the EKP inequality \eqref{EKP-ineq}.}

Our proof is different from Kadyrov's, and it produces sharper bounds.
We show in Theorem~\ref{t.optimal-EKP} that for every $\eps>0$ and a H\"older continuous $\psi:\Sigma^+\to\R$, there exists a $\delta>0$ such that if $h_\mu(\sigma)$ is $\delta$-close to $h_{\mu_0}(\sigma)$, then
\begin{equation}\label{sharp-EKP}
\left|\int\psi d\mu_0-\int\psi d\mu\right|\leq e^\eps\sqrt{2}\sigma_{\mu_0}(\psi)\sqrt{h_{\mu_0}(\sigma)-h_\mu(\sigma)}.
\end{equation}
Here $\sigma_{\mu_0}^2(\psi)$ is the {\em asymptotic variance} of $\psi$ with respect to $\mu_0$, see \eqref{e.asymp-var}.
The square root and the constant $\sqrt{2}\sigma_{\mu_0}(\psi)$ are sharp (Theorem~\ref{t.optimal-EKP}).
The sharpness  of the square root is an answer to a question of Kadyrov  \cite[p.\ 240]{Kadyrov-Effective-Uniqueness}. It is instructive to compare \eqref{EKP-ineq} and \eqref{sharp-EKP} in the special case when $\psi\not\equiv 0$, and $\psi=u-u\circ\sigma+\text{const}$ with $u$ bounded and measurable. In this case the right-hand-side of \eqref{sharp-EKP} is zero (which is sharp), and the right-hand-side of \eqref{EKP-ineq} is positive (which is not sharp).

The measure of maximal entropy $\mu_0$  is the equilibrium measure of the zero potential. Our results extend to equilibrium measures $\mu_\phi$ of other potentials $\phi$.
Suppose $\Sigma^+$ has finite Gurevich entropy, $\sup\phi<\infty$ and  $\phi$ is weakly H\"older continuous (see \S\ref{s.setup}). In Theorem~\ref{t.suboptimal-EKP}, we show that if $\phi$ is SPR in the sense of  \cite{Sarig-CMP-2001}, then for every shift invariant probability measure $\mu$ and  $\beta$-H\"older continuous function $\psi$,
\begin{equation}\label{EKP-ineq-pressure}
\left|\int\psi d\mu-\int\psi d\eqp\right|\leq C_{\phi,\beta}\|\psi\|_\beta\sqrt{P_{\eqp}(\phi)-P_\mu(\phi)},
\end{equation}
where  $P_\nu(\phi):=h_\nu(\sigma)+\int\phi d\nu$. A sharp version similar to \eqref{sharp-EKP} holds as well.

The SPR property is a necessary condition for \eqref{EKP-ineq-pressure}: In Corollary~\ref{cor.EKP-SPR}, we show that   \eqref{EKP-ineq-pressure} fails whenever $\phi$ is not SPR. The  case $\phi\equiv 0$ follows from \cite{Ruette}.

\medskip
Let us compare  Kadyrov's proof to our proof.
Kadyrov's proof is better  in two ways:
It is much shorter, and it yields a finitary version of \eqref{EKP-ineq} with $\frac{1}{n}H_{\mu}(\alpha_0^{n-1})$ replacing $h_\mu(\sigma)$, see \cite{kadyrov2017effective}.
The reader  may wonder why we needed a different proof.
One reason is that our proof gives  sharper, optimal, bounds. But there is another reason, of a more technical nature, which we  would like to explain.

Assume $\sigma:\Sigma^+\to\Sigma^+$ is topologically mixing, and let $\widehat{T}$ denote the transfer operator\footnote{
In the notation of  \S\ref{s.TDF-for-SPR} in this paper,  $\widehat{T} f=\lambda_0^{-1}M_{h_0}^{-1}L_0 M_{h_0}$.
} of the measure of maximal entropy. Let $(\mathcal H_\beta,\|\cdot\|_\beta)$ denote the space of $\beta$-H\"older continuous functions as defined in \eqref{eq:H-beta}.
In the case of finite alphabets, $\widehat{T}$ has spectral gap when acting on   $(\mathcal H_\beta,\|\cdot\|_\beta)$, and since
$
\|\cdot\|_{\beta}\geq \|\cdot\|_\infty
$, it follows that
$
\|\widehat{T}^n f-\int f d\mu_0\|_{\infty}\to 0 \text{ exponentially fast}.
$
This exponential {\em uniform} convergence seems to us to  be crucial for Kadyrov's proof, see  \cite[pp.\ 244-245]{Kadyrov-Effective-Uniqueness}.

But in the case of an infinite alphabet, we cannot expect uniform exponential convergence like that for all H\"older continuous functions, even in the SPR case. Let $\mathfs G$ be the graph associated to the shift (see \S\ref{s.setup}), and suppose every vertex in $\mathfs G$ has finite  degree. If $f$ is the indicator of a cylinder, then $\widehat{T}^n f$ vanishes outside a finite union of partition sets (which depends on $n$). So $\|\widehat{T}^n f-\int f d\mu_0\|_\infty\geq |\int f d\mu_0|$ for all $n$, and  $\|\widehat{T}^n f-\int f d\mu_0\|_\infty\not\to 0$.

\medskip
\noindent
This is the obstacle that forced us to seek a different proof.

\medskip
\noindent
There is an important  class of countable Markov shifts which do have Banach spaces with spectral gap so that $\|\cdot\|_{\mathcal L}\geq\|\cdot\|_\infty$:  The shifts with the {\em big images and pre-images (BIP) property} (\cite{Aaronson-Denker}, \cite{Sarig-Gibbs}, see also \S\ref{s.eq}, Example 2). In the infinite alphabet BIP case, all measures of maximal entropy  have infinite entropy, and  there is no hope to get \eqref{EKP-ineq}.  But there may still be unbounded  potentials  with equilibrium measures with finite pressure.  For the EKP inequality for such measures, in the form \eqref{EKP-ineq-pressure}, see \cite{Ruhr-BIP}.

\section{A review of the theory of topological Markov shifts}\label{s.setup}
\subsection{Topological Markov shifts}\label{s.TMS} Let $\mathfs G$ denote a countable directed graph with set of vertices $S$ and set of directed edges $E$. If there is an edge from $a$ to $b$, we write $a\to b$.
Set
$$
\Sigma^+=\Sigma^+(\mathfs G):=\{\un{x}=(x_0,x_1,\ldots): x_i\in S,  x_i\to x_{i+1}\text{ for all }i\}.
$$
For every $\un{x}\neq \un{y}$ in $\Sigma^+$, let $t(\un{x},\un{y})=\min\{i:x_i\neq y_i\}$.  We  equip $\Sigma^+$ with the metric
\begin{equation}\label{eq.metric}
d(\un{x},\un{y}):=\begin{cases}
0 & \un{x}=\un{y}\\
\exp[-t(\un{x},\un{y})] & \text{otherwise}.
\end{cases}
\end{equation}
\begin{definition}
The topological dynamical system $\sigma:\Sigma^+\to\Sigma^+$ given by
$ \sigma(\un{x})_{i}=x_{i+1}
$  is called the {\em one-sided topological Markov shift (TMS)} associated to $\mathfs G$. The elements of $S$ are called {\em states}, and $\sigma$ is called the {\em left shift}.
\end{definition}
\noindent
When $|S|=\aleph_0$, we will also call $\Sigma^+$ a {\em countable state Markov shift}.

The sets
$[\un{a}]=
[a_0,\ldots,a_{n-1}]:=\{\un{x}\in\Sigma^+:x_i=a_i\ (i=0,\ldots,n-1)\}
$
$(\un{a}\in\bigcup_{n}S^n)$ are called {\em cylinders} of length $n$. Cylinders of length one are also called {\em partition sets}.
The cylinders form a basis for the topology, and they generate the Borel $\sigma$-algebra, which we denote by $\mathfs B$.

\subsection{Topological transitivity and topological mixing}\label{s.spectral-decomposition}
We write $a\xrightarrow[]{n}b$ when there is a non-empty cylinder of the form $[a,\xi_1,\ldots,\xi_{n-1},b]$. In particular  $a\xrightarrow[]{1}b\Leftrightarrow a\to b$. The following simple facts are well-known:
\begin{enumerate}[(1)]
\item $\sigma:\Sigma^+\to\Sigma^+$ is topologically transitive if and only if $\forall a,b\in S$ $\exists n$ such that $a\xrightarrow[]{n}b$. Equivalently, $\mathfs G$ is {\em strongly connected}: $\forall(a,b)\in S^2$, there is  a path from $a$ to $b$.
\item   If $\sigma:\Sigma^+\to\Sigma^+$ is topologically transitive, then
    $
    p_a:=\gcd\{n: a\xrightarrow[]{n}a \}
    $ $(a\in S)$
are all equal to the same value $p\geq 1$, and $p$ is called the {\em period} of $\Sigma^+$.

\item $\sigma:\Sigma^+\to\Sigma^+$ is topologically mixing if and only if it is topologically transitive, and its period is equal to one.

\item {\bf The spectral decomposition:}  Suppose $\Sigma^+$ is a topologically transitive TMS with period $p>1$, then we can decompose
    $
    \Sigma^+=\Sigma_0^+\uplus \Sigma_1^+\uplus\cdots\uplus\Sigma_{p-1}^+
    $
    where $\sigma(\Sigma_i^+)=\Sigma_{i+1\,\mathrm{mod}\, p}^+$, and where $\sigma^p:\Sigma_i^+\to\Sigma_i^+$ are all topologically conjugate to a topologically {\em mixing} countable Markov shift.

    Briefly, this is done as follows.  There is an equivalence relation on the states of $\Sigma^+$ given by  $a\sim b\Leftrightarrow a=b\text{ or }a\xrightarrow{n}b$ for some $n$ divisible by $p$. There are $p$ equivalence classes $S_0,\ldots, S_{p-1}$, and $\Sigma_i^+:=\{\un{x}\in\Sigma^+: x_0\in S_i\}$. The map $\sigma^p:\Sigma_i^+\to\Sigma_i^+$ is topologically conjugate to the TMS with set of states $\{[a_0,\ldots,a_{p-1}]:a_0\in S_i\}\setminus\{\emptyset\}$ and edges  $[\un{a}]\to[\un{b}]$ when $a_{p-1}\to b_0$, and this TMS is topologically mixing.
\end{enumerate}
\noindent
The spectral decomposition is a tool for reducing statements on topologically transitive TMS to the topologically mixing case. We will use this tool frequently.


\subsection{Weak H\"older continuity and summable variations}
The  {\em $n^{\mathrm{th}}$ oscillation} (aka the {\em $n^{\mathrm{th}}$ variation}) of a function $\phi:\Sigma^+\to\R$ is
$$
\var_n(\phi):=\sup\{|\phi(\un{x})-\phi(\un{y})|:x_i=y_i \ (i=0,\ldots,n-1)\}.
$$

A function  $\phi:\Sigma^+\to\R$ is called {\em $\theta$-weakly H\"older continuous} if $\theta\in (0,1)$ and there exists $A>0$ such that $\var_n(\phi)\leq A\theta^n$ for all $n\geq 2$.
This condition does not imply that $\phi$ is bounded, and the choice $n\geq 2$ is done to include all functions of the form $\phi(\un{x})=\phi(x_0,x_1)$. A bounded $\theta$-weakly H\"older continuous is H\"older continuous  with exponent $\beta:=-\log\theta$ with respect to the metric \eqref{eq.metric}. We define the space of such functions in \eqref{eq:H-beta}.

Some of our results hold under the following weaker regularity assumption, called {\em summable variations}: $\sum_{n\geq 2}\var_n(\phi)<\infty$.

\subsection{Pressure and equilibrium measures}\label{s.pressure}
Suppose $\Sigma^+$ is topologically mixing,  $\phi:\Sigma^+\to\R$ has summable variations and let $\phi_n:=\sum_{k=0}^{n-1}\phi\circ\sigma^k$. Given $a\in S$, let
$$
P_G(\phi):=\lim_{n\to\infty}\frac{1}{n}\log Z_n(\phi,a),\text{ where }Z_n(\phi,a):=\sum_{\sigma^n(\un{x})=\un{x}}e^{\phi_n(\un{x})}1_{[a]}(\un{x})\label{Z_n}.
$$
The limit exists and is independent of $a$, see \cite{Sarig-ETDS-99}.\footnote{This reference states the result under stronger regularity assumptions on $\phi$, but the proofs there work verbatim for functions with  summable variations.\label{footnote-2}} It is always bigger than $-\infty$, but it could be equal to $+\infty$.
\begin{definition}
 $P_G(\phi)$ is called the  {\em Gurevich pressure} of $\phi$. The {\em Gurevich entropy} of $\Sigma^+$ is
$\displaystyle
h:=P_G(0)=\lim_{n\to\infty}\frac{1}{n}\log\#\{\un{x}\in\Sigma^+:x_0=a\ , \sigma^n(\un{x})=\un{x}\}
$.
\end{definition}

Let $\mathfs M(\Sigma^+)$ denote the collection of all $\sigma$-invariant Borel probability measures on $\Sigma^+$, and let  $h_\mu(\sigma)$ denote the metric entropy of $\mu\in\mathfs M(\Sigma^+)$.
{\em Gurevich's variational principle} \cite{Gurevich-Entropy} says that if $\Sigma^+$ is topologically mixing, then
$$
h:=P_G(0)=\sup\left\{h_\mu(\sigma):\mu\in\mathfs M(\Sigma^+)\right\}.
$$
Any measure which achieves the supremum is called a {\em measure of maximal entropy}. Such measures do not always exist, but if they do, and if $\Sigma^+$ is topologically transitive,  then they are unique \cite{Gurevich-Measures-Of-Maximal-Entropy}. We will describe the structure of the measure of maximal entropy in the following section.

The Gurevich pressure of a general $\phi$ with summable variations satisfies a similar  variational principle, which we now explain.

 A measurable function $\phi$ is called {\em one-sided $\mu$-integrable} if at least one of the integrals $\int_{[\phi>0]}\phi d\mu,
\int_{[\phi<0]}\phi d\mu$ is finite. Let
$$
\mathfs M_\phi(\Sigma^+):=\left\{\mu\in\mathfs M(\Sigma^+):\begin{array}{l}
\text{$\phi$ is one-sided $\mu$-integrable, and }\\
\text{$(h_\mu(\sigma),\int\phi d\mu)\neq (+\infty,-\infty)$}
\end{array}
\right\}.
$$
This is the collection of  $\mu \in \mathfs M(\Sigma^+)$ for which the  expression
$$
P_\mu(\phi):=h_\mu(\sigma)+\int\phi d\mu,
$$
is well-defined (we allow $P_\mu(\phi)=\pm\infty$ but forbid $P_\mu(\phi)=\infty-\infty$).

In this paper we will be mostly interested in the case when  $\Sigma^+$ has finite Gurevich entropy and  $\sup\phi<\infty$. In this case $P_\mu(\phi)$ is well-defined for all $\mu\in \mathfs M(\Sigma^+)$, and $\mathfs M_\phi(\Sigma^+)=\mathfs M(\Sigma^+)$.

The {\em variational principle} for TMS   states that if $\Sigma^+$ is topologically mixing and if $\phi$ has summable variations, then
\begin{equation}\label{eq:measure_pressure}
	P_G(\phi)=\sup\left\{ P_\mu(\phi):\mu\in\mathfs M_\phi(\Sigma^+) \right\}.
\end{equation}
See \cite{Sarig-ETDS-99} for the special case  $\sup\phi<\infty$,
and \cite{Iommi-Jordan-Todd-Suspension-Flows} in general.

A measure $\mu\in\mathfs M_\phi(\Sigma^+)$ which achieves the supremum in \eqref{eq:measure_pressure} is called  an {\em equilibrium measure for the  potential} $\phi$. The equilibrium measures for the constant potential are the measures of maximal entropy.

\medskip
So far we have only discussed the  topologically mixing case. In the topologically transitive case, with period $p$, we define
$$
P_G(\phi):=(1/p)P_G(\phi_p|_{\Sigma_0^+}),
$$ where $\phi_p:=\sum_{i=0}^{p-1}\phi\circ\sigma^i
$, and $\Sigma_0^+$ is some (any) of the components in the spectral decomposition of $\Sigma^+$.

With this definition, the variational principle holds, and $m_0$ is an equilibrium measure for $\phi_p|_{\Sigma_0^+}$ if and only if $m:=\frac{1}{p}\sum_{i=0}^{p-1}m_0\circ\sigma^{-i}$ is an equilibrium measure for $\phi$. Furthermore,  in this case $m_0:=m(\cdot|\Sigma_0^+)$, the conditional measure on $\Sigma^+_0$. For more details, see the end of the proof of Theorem \ref{t.optimal-EKP}.

\subsection{Existence and  structure of equilibrium measures}\label{s.eq}
The topic is intimately related to the eigenvector problem for {\em Ruelle's operator}
\begin{equation}\label{eq:Ruelle-Operator}
(L_\phi f)(\un{x}):=\sum_{\sigma(\un{y})=\un{x}}e^{\phi(\un{y})}f(\un{y}).
\end{equation}
We recall the connection (\cite{Bowen-LNM,Sarig-Null,Buzzi-Sarig}).

 Fix a state $a\in S$ and $\un{x}\in \bigcup_{n>0}\sigma^{-n}[a]$, let  $\tau_a(\un{x}):=1_{[a]}(\un{x})\min\{n>0: x_n=a\}$.  Given a function $\phi:\Sigma^+\to\R$, let
$ \displaystyle Z_n^\ast(\phi,a):=\sum_{\sigma^n(\un{x})=\un{x}}e^{\phi_n(\un{x})}1_{[\tau_a=n]}(\un{x})$.
Recall that  $\phi_n=\sum_{i=0}^{n-1}\phi\circ\sigma^i$, and  $\displaystyle Z_n(\phi,a):=\sum_{\sigma^n(\un{x})=\un{x}}e^{\phi_n(\un{x})}1_{[a]}(\un{x}).$
\begin{definition}
Suppose $\phi$ is a function with summable variations and finite Gurevich pressure on a topologically mixing TMS $\Sigma^+$. We say that $\phi$ is {\em positively recurrent}, if for some state $a$,
$$
\sum_{n=1}^\infty \lambda^{-n}Z_n(\phi,a)=\infty\text{ and }
\sum_{n=1}^\infty n\lambda^{-n}Z_n^\ast(\phi,a)<\infty,\text{ where }\lambda:=\exp P_G(\phi).
$$
\end{definition}
\noindent
(This should not be confused with {\em strong} positive recurrence, a condition that is discussed in the next section.)

The \textit{Generalized Ruelle's Perron-Frobenius theorem} \cite{Sarig-Null}\footnote{See the footnote on page \pageref{footnote-2}.} states that if $\Sigma^+$ is a topologically mixing TMS and $\phi$ has summable variations and finite Gurevich pressure, then $\phi$ is positively recurrent if and only if  there is a positive continuous function $h_\phi:\Sigma^+\to \R_+$ and a $\sigma$-finite measure $\nu_\phi$
such that
\begin{equation}
\begin{aligned}
	\label{def:GRPF}
&	L_\phi h_\phi=e^{P_G(\phi)}h_\phi, \quad L_\phi^*\nu_\phi =e^{P_G(\phi)}\nu_\phi, \quad \int h_\phi d\nu_\phi=1\text{ and }\\
&e^{-n P_G(\phi)}(L_\phi^n 1_{[\un{a}]})(\un{x})\xrightarrow[n\to\infty]{}h_\phi(\un{x})\nu_{\phi}[\un{a}]\text{ pointwise for all cylinders $[\un{a}]$}.
\end{aligned}
\end{equation}
In this case $h_\phi$ is continuous,  bounded away from zero and infinity on partition sets, and $\nu_\phi$ gives finite and positive measure to every non-empty cylinder.
The measure $m_\phi$ defined as $dm_\phi= h_\phi d\nu_\phi$ turns out to be a $\sigma$-invariant probability measure, and is called the \textit{Ruelle-Perron-Frobenius (RPF) measure}.

\begin{theorem}[\cite{Buzzi-Sarig}, \cite{Cyr-Sarig}]\label{t.eq-measure}
Let $\phi$ be a potential with summable variations and finite Gurevich pressure on a topologically mixing TMS, and suppose $\sup\phi<\infty$.
\begin{enumerate}[(1)]
\item If $\phi$ admits an equilibrium measure $m$, then this measure is unique, $\phi$ is positively recurrent, and $m$ equals the RPF measure of $\phi$.
\item Conversely, if $\phi$ is positively recurrent and the RPF measure $m_\phi$ has finite entropy, then $m_\phi$ is the unique equilibrium measure of $\phi$.
\end{enumerate}
\end{theorem}

One obtains the following corollary, which is due to Gurevich \cite{Gurevich-Measures-Of-Maximal-Entropy} for $\phi=0$.
\begin{corollary}\label{cor.PR-eq-measure}
Let $\phi$ be a potential with summable variations and finite Gurevich pressure on a topologically mixing TMS with finite Gurevich entropy, and suppose $\sup\phi<\infty$. Then $\phi$ has an equilibrium measure if and only if $\phi$ is positively recurrent. In this case, the equilibrium measure is unique.
\end{corollary}

An important consequence of the description of the equilibrium measure as the RPF measure is the following identity for  conditional probabilities \cite{Ledrappier-g-measures,Walters-g-measures}:

\begin{equation}\label{eq.DLR}
\E_{m_\phi}(f|\sigma^{-n}\mathfs B)=\lambda^{-n}(h_\phi^{-1}L_\phi^n(h_\phi f))\circ\sigma^n\ \ \ m_\phi\text{-a.e.}
\end{equation}

\subsection*{Example 1 (SFT)} In the finite alphabet case (when $\Sigma^+$ is a topologically mixing subshift of finite type), every $\phi$ with summable variations is  positively recurrent and the RPF measure is always an equilibrium measure. This is a consequence of {\em Ruelle's Perron-Frobenius theorem}, see \cite{Bowen-LNM}.

\subsection*{Example 2 (BIP)} A topologically mixing Markov shift is said to have the {\em big images and pre-images (BIP) property} if there is a finite set of states $b_1,\ldots,b_N$ such that for every state $a$ there are some edges $a\to b_i, b_j\to a$. On a TMS with the BIP property, every $\phi$ with summable variations such that $\var_1(\phi)<\infty$ and $P_G(\phi)<\infty$ is positively recurrent \cite{Sarig-Gibbs}, \cite{Mauldin-Urbanski}. The corresponding RPF measure is an equilibrium measure if and only if it belongs to $\mathfs M_\phi(\Sigma^+)$.

Sometimes this is not the case.
Fix a probability vector $\vec{p}=(p_i)_{i\in\N\cup\{0\}}$ with infinite entropy, and  suppose $\Sigma^+=\N^{\N\cup\{0\}}$ and $\phi(\un{x})=\log p_{x_0}$. Clearly, $P_G(\phi)=0$,  $\Sigma^+$ has the  BIP property, and the RPF measure is the Bernoulli measure $\mu$ with probability vector $\vec{p}$. For this measure $h_\mu(\sigma)+\int\phi d\mu$ is not well-defined, because  $h_\mu(\sigma)=+\infty$ and $\int\phi d\mu=-\infty$.

\subsection*{Example 3 (Measures of maximal entropy \cite{Parry-Intrinsic-MC}, \cite{Gurevich-Measures-Of-Maximal-Entropy}).}
Returning to the countable alphabet case, let's consider the special case of measures of maximal entropy, $\phi\equiv 0$.
Let  $h:=P_G(0)$ be the Gurevich entropy, and set $\lambda:=\exp(h)$.

It is not difficult to check using the identity $h_\phi=\nu_\phi[a]^{-1}\lim_{n\to\infty} \lambda^{-n}L_\phi^n 1_{[a]}$ that if $\phi\equiv 0$,
  then  $h_\phi(\un{x})$ depends only on $x_0$,
  so there is a positive vector $\vec{\ell}=(\ell_a)_{a\in S}$ such that
$h_\phi(\un{x})=\ell_{x_0}$. Since $L_\phi h_\phi=\lambda h_\phi$,
$
\sum_{a\in S}  \ell_a T_{ab} =\lambda \ell_b
$, where $T_{ab}=1$ when $a\to b$ and $T_{ab}=0$ if $a\not\to b$.
So $\vec{\ell}$ is a left eigenvector of the transition matrix $T$.
Similarly, the vector $\vec{r}=(r_a)_{a\in S}$ given by $r_a:=\nu_\phi[a]$ is a right eigenvector of $T$, because
$r_a=\nu_\phi[a]=\lambda^{-1}\nu_\phi(L_\phi 1_{[a]})=\lambda^{-1}\nu_\phi(\sigma[a])=\lambda^{-1}\sum_{b\in S} T_{ab} r_b$.

 Let $m_\phi=h_\phi \nu_\phi$ be the RPF measure.

 We write $m_\phi[x_0,\dots,x_n]=m_\phi([x_0,\dots,x_n])$,
 $m_\phi(x_0| x_1,\dots, x_n) = \frac{m_\phi([x_0,\dots,x_n])}{m_\phi(\sigma^{-1}[x_1,\dots,x_n])}$
 and
 $m_\phi(x_0|x_1,\dots):=\E_{m_\phi}(1_{[x_0]}|\sigma^{-1}\mathfs{B})(\un x)$.
 Substituting $f=1_{[x_i]}$, $n=1$ in  \eqref{eq.DLR} and evaluating at $\sigma^i(\un x)$, we obtain
$$
m_\phi(x_i|x_{i+1},\ldots)=\frac{T_{x_i, x_{i+1}}\ell_{x_{i}}}{\lambda \ell_{x_{i+1}}}\ \ m_\phi\text{-a.e.}
$$
Call the value on the right-hand side $q_{x_i,x_{i+1}}$. Since it is  independent of $x_j$ for $j>i+1$, $
m_\phi(x_i|x_{i+1},\ldots,x_{n-1})=q_{x_i,x_{i+1}}
$ for all $n>i+1$.
Thus, by the shift invariance of $m_\phi$,
\begin{equation}\label{ratio}
\frac{m_\phi[x_i,x_{i+1},\ldots,x_{n-1}]}{m_\phi[x_{i+1},\ldots,x_{n-1}]}=m_\phi(x_i|x_{i+1},\ldots,x_{n-1})=q_{x_i,x_{i+1}} .
\end{equation}
Taking the product of \eqref{ratio} over $0\leq i\leq n-2$, we
obtain
$$
m_\phi[x_0,\ldots,x_{n-1}]=q_{x_0,x_1}\cdots q_{x_{n-2},x_{n-1}}m_\phi[x_{n-1}].
$$
Equivalently, if $p_a:=m_\phi[a]\equiv\ell_a r_a$, and $p_{ab}:=q_{ab}\frac{m_\phi[b]}{m_\phi[a]}\equiv \frac{T_{ab}r_b}{\lambda r_a}$, then
\begin{equation}\label{eq.Parry}
m_\phi[x_0,\ldots,x_{n-1}]=p_{x_0} p_{x_0 x_1}\cdots p_{x_{n-2} x_{n-1}}\ \ \text{(``Parry's measure")}.
\end{equation}

\subsection{Strong positive recurrence}\label{s.SPR-def}
This is a strengthening of the positive recurrence condition, which implies that $L_\phi$ acts quasi-compactly on a ``sufficiently rich" Banach space of functions (see the next section).

We begin with the special case when $\Sigma^+$ is topologically mixing, and $\phi\equiv 0$.
 Fix a state $a$, and let $f_n(a)$ denote the number of first return loops with length $n$, at  $a$:
$$f_n(a):=\#\{\un{x}\in\Sigma^+: \sigma^n(\un{x})=\un{x},\; x_0=a,\;  x_1,\ldots,x_{n-1}\neq a\}.$$ Let $F_a(z):=\sum_{n\geq 1}f_n(a) z^n$ be the associated generating function.

\begin{definition}
A topologically mixing TMS $\Sigma^+$ with finite Gurevich entropy is called {\em strongly positively recurrent (SPR)}, if for some state $a$,  $F_a(R_a)>1$ where  $R_a:=$ radius of convergence of $F_a(z)$.
\end{definition}
\noindent
 It follows from  \cite{Vere-Jones-Geometric-Ergodicity} that (a) SPR implies that $\phi\equiv 0$ is positively recurrent; (b) the SPR property is independent of $a$: if $F_a(R_a)>1$ for some $a$, then $F_a(R_a)>1$ for all $a$;  and (c) SPR implies that Parry's measure is an exponentially mixing Markov chain.

\medskip
The SPR condition was extended to general potentials with summable variations in \cite{Sarig-CMP-2001} (see \cite{Gurevich-Savchenko} for the special case when $\var_2(\phi)=0$, i.e.\ $\phi$ is \textit{Markovian}).

 Suppose $\Sigma^+$ is topologically mixing, and $a$ is a state. The {\em induced map on $[a]$} is the map $\sigma^{\tau_a}:[a]'\to[a]'$, where $[a]':=\{\un{x}\in [a]:x_i=a\text{ infinitely often}\}$, and $$\tau_a(\un{x}):=1_{[a]}(\un{x})\min\{n>0:x_n=a\}.$$
  This map has a  coding as a full shift: Let $
\ov{S}:=\{[a,\xi_1,\ldots,\xi_n,a]:\xi_i\neq a\}\setminus\{\emptyset\}
$, $\ov{\Sigma}^+:=\ov{S}^{\N\cup\{0\}}$ and define $\pi:\ov{\Sigma}^+\to [a]'$ by $$\pi([a,\un{\xi}_1,a],[a,\un{\xi}_2,a],\ldots):=(a,\un{\xi}_1,a,\un{\xi}_2,a,\ldots).$$ Then $\pi^{-1}\circ\sigma^{\tau_a}\circ\pi$ is the left shift on $\ov{\Sigma}^+$.
Given a function $\phi:\Sigma^+\to\R$, we let
$$
\ov{\phi}:=(\sum_{k=0}^{\tau_a-1}\phi\circ\sigma^k)\circ\pi.
$$
If $\phi$ is weakly H\"older continuous on $\Sigma^+$, then $\ov{\phi}$ is weakly H\"older continuous on $\ov{\Sigma}^+$.

\begin{definition}\label{def:Disc}
Suppose $\phi$ is a weakly H\"older continuous potential with finite Gurevich pressure on a topologically mixing TMS, and let $a$ be a state.
The {\em discriminant} of $\phi$ at state $a$  is the (possibly infinite) expression
$$\Delta_a[\phi]:=\sup\{P_G(\ov{\phi+p}):p\in\R\text{ such that }P_G(\ov{\phi+p})<\infty\}.$$
We say that $\phi$ is {\em strongly positively recurrent (SPR)} if $\Delta_a[\phi]>0$ for some $a$.
\end{definition}
\noindent
The definition extends to $\phi$ with summable variations, but some care is needed since $\ov{\phi}$ may not have summable variations in this case, see  \cite[Lemma~2]{Sarig-CMP-2001}, which guarantees that $P_G(\ov{\phi})$ is still well defined.

\begin{lemma}\label{l.p-star}
If $\Sigma^+$ is topologically mixing,  $\phi$ has summable variations, and $P_G(\phi)<\infty$, then $\phi$ is strongly positively recurrent if and only if for some state $a$,
$$
P^*_G(\phi,a):=\limsup_{n\to\infty}\frac{1}{n}\log Z_n^\ast(\phi,a) < \lim_{n\to\infty}\frac{1}{n}\log Z_n(\phi,a)=P_G(\phi).
$$
If this happens for some state, then it happens for all states.
\end{lemma}
\begin{proof}
This is a consequence of the discriminant theorem of \cite{Sarig-CMP-2001}.
\end{proof}

\subsection*{Example} Consider the special case $\phi\equiv 0$. In this case $\ov{\phi+p}=\ov{p}=p\cdot \tau_a\circ \pi$. So for every $\ov{a}:=[a,\un{\xi},a]\in\ov{S}$,
$$
Z_n(\ov{\phi+p},\ov{a})=\sum_{[a,\un{\xi}_1,a],\ldots,[a,\un{\xi}_{n-1},a]\in \ov{S}}
\hspace{-0.5cm}e^{p(|a\un{\xi}|+|a\un{\xi}_1|+\cdots+|a\un{\xi}_{n-1}|)}=e^{p|a\un{\xi}|}\biggl(\sum_{[a\un{\eta}a]\in\ov{S}}
e^{p|a\un{\eta}|}\bigg)^{n-1},
$$
and $P_G(\ov{\phi+p})=\log \sum_{[a\un{\eta}a]\in \ov{S}}e^{p|a\un{\eta}|}$. The terms with $|a\un{\eta}|=k$ represent first return time loops with length $k$ at the vertex $a$. So
$P_G(\ov{\phi+p})=\log\sum_{k=1}^\infty f_k(a)e^{kp}$. It follows that
$\Delta_a[0]=\log \sum_{n=1}^\infty f_n(a) R^n_a$, where $R_a$ is the radius of convergence of $F_a(z)=\sum_{n\geq 1} f_n(a) z^n$. So $\phi\equiv 0$ is SPR if and only if $F_a(R_a)>1$.

\medskip
It is shown in \cite{Sarig-CMP-2001} that if $\Sigma^+$ is topologically mixing,  $\phi$ has summable variations, and $P_G(\phi)<\infty$, then: (a) SPR implies positive recurrence; (b) The SPR property is independent of the choice of $a$: If $\Delta_a[\phi]>0$ for some $a$, then $\Delta_a[\phi]>0$ for all $a$.
Additionally, in \cite{Cyr-Sarig} it is shown that: (c) If $\phi$ is weakly H\"older continuous, then the RPF measure of $\phi$ has exponential decay of correlations for H\"older continuous test functions.

\medskip
So far we have only discussed the topologically mixing case. If $\Sigma^+$ is topologically transitive with period $p>1$, and $\phi$ is a potential with summable variations and finite Gurevich pressure, then we say that $\phi$ is {\em strongly positively recurrent} if and only if $\phi_p|_{\Sigma_0^+}$ is strongly positively recurrent, where $\phi_p:=\sum_{i=0}^{p-1}\phi\circ\sigma^i$ and $\Sigma_0^+$ is some (any) component of the spectral decomposition.

\subsection{SPR and spectral gap}\label{s.SGP}
Let $\Sigma^+$ be a topologically mixing TMS.
A $\theta$-weakly H\"older continuous function $\phi$ is said to have the {\em spectral gap property} if  there exists a Banach space $(\mathcal L,\|\cdot\|_{\mathcal L})$ of functions $f:\Sigma^+\to\C$ with the following properties:
\begin{enumerate}[(a)]
\item $(L_\phi f)(\un{x})=\sum_{\sigma(\un{y})=\un{x}}e^{\phi(\un{y})}f(\un{y})$ converges absolutely whenever $f\in\mathcal L$, and $\mathcal L$ contains all the indicators of cylinder sets.
\item $f\in\mathcal L\Rightarrow |f|\in\mathcal L$ and $\bigl\||f|\bigr\|_{\mathcal L}\leq \|f\|_{\mathcal L}$.
\item $\mathcal L$-convergence implies uniform convergence on cylinders.
\item $L_\phi(\mathcal L)\subset\mathcal L$ and $L_\phi:\mathcal L\to \mathcal L$ is bounded.
\item $L_\phi=\lambda P+N$ where $\lambda=\exp P_G(\phi)$,  $P,N$ are bounded linear operators on $\mathcal L$ such that  $PN=NP=0$, $P^2=P$, $\dim \mathrm{Image}(P)=1$, and the spectral radius of $N$ is strictly less than $\lambda$.
\item For every bounded $\theta$-weakly H\"older continuous function $\psi:\Sigma^+\to\R$, $z\in\C$, $L_{\phi+z\psi}$ is bounded on $\mathcal L$, and $z\mapsto L_{\phi+z\psi}$ is analytic on some complex neighborhood of zero.
  (See \cite[Chapter VII \S 1]{Kato-Book} for the definition of analyticity for families of operators depending on a complex parameter.)
\end{enumerate}
Property (e) says that the spectrum of $L_\phi:\mathcal L\to \mathcal L$ consists of a simple eigenvalue $\lambda$ and a compact subset of $\{z:|z|<\lambda\}$. $P$ is the eigenprojection of $\lambda$. The ``spectral gap" is the difference between $|\lambda|$ and the spectral radius of $N$.

The  paper  \cite{Cyr-Sarig} proves that {\em a weakly H\"older continuous function $\phi$ with finite Gurevich pressure on a topologically  mixing TMS has the spectral gap property, if and only if $\phi$ is strongly positively recurrent.}

The space $\mathcal L$ constructed  there has the following additional property \cite[p. 650]{Cyr-Sarig}. Let $\mathcal H_\beta$ be the space of $\beta$-H\"older continuous functions,
\begin{equation}\label{eq:H-beta}
\mathcal H_\beta:=\left\{\psi:\Sigma^+\to\R:
\|\psi\|_\beta:=\|\psi\|_\infty+\sup_{\un{x}\neq \un{y}}\frac{|\psi(\un{x})-\psi(\un{y})|}{e^{-\beta t(\un x,\un y)}}<\infty
\right\}.
\end{equation}
\begin{enumerate}[(g)]
\item   If $\theta=e^{-\beta}$, then for all $f\in\mathcal H_\beta$ and $g\in\mathcal L$, $fg\in\mathcal L$ and $\|fg\|_{\mathcal L}\leq \|f\|_\beta\|g\|_{\mathcal L}$.
\end{enumerate}

We caution the reader that the SPR property does {\em not} imply the spectral gap property for topologically transitive TMS with period $p>1$. We may still find a Banach space with (a)--(d) on which $L_\phi$ acts quasi-compactly, but there will be $p$ points in the spectrum with modulus $\exp P_G(\phi)$, and not just one as in (e).

\section{The pressure function}
Throughout this section, we fix a topologically mixing one-sided countable Markov shift $\sigma:\Sigma^+\to\Sigma^+$ with finite positive Gurevich entropy $h$,
 and potentials $\phi,\psi$ with summable variations such that $\sup\phi<\infty$ and $\|\psi\|_\infty<\infty$. It follows that
 $$
 \mathfs M_{\phi+t\psi}(\Sigma^+)=\mathfs M(\Sigma^+)\text{ for all }t.
 $$
Recall the definition of the Gurevich pressure from \S\ref{s.pressure}.
\begin{definition}
The  {\em pressure function of $\phi$ in direction $\psi$} is the function $$\p_{\phi,\psi}:\R\to\R\ , \ \p_{\phi,\psi}(t):=P_G(\phi+t\psi).$$
\end{definition}
\noindent
By the variational principle \eqref{eq:measure_pressure},
\begin{align*}
\p_{\phi,\psi}(t)&
=\sup\{P_\mu(\phi)+t\int\psi d\mu: \mu\in\mathfs M(\Sigma^+)\},
\end{align*}
where  $
 P_\mu(\phi):=h_\mu(\sigma)+\int\phi d\mu.
 $

\begin{theorem}\label{t.pressure-convex-prop}
Let $\Sigma^+$ be a topologically mixing TMS with finite Gurevich entropy. If $\phi,\psi$ have summable variations, $\sup\phi<\infty$ and  $\|\psi\|_\infty<\infty$, then
\begin{enumerate}[(1)]
\item $\mathfrak p_{\phi,\psi}(t)$ is finite, convex, and continuous on $\R$.
\item $\mathfrak p_{\phi,\psi}(t)$ has well-defined finite one-sided derivatives
$$(D^\pm\mathfrak p_{\phi,\psi})(t):=\lim_{h\to 0^\pm }\frac{1}{h}[\mathfrak p_{\phi,\psi}(t+h)-\mathfrak p_{\phi,\psi}(t)].$$
\item
$
\displaystyle\lim_{t\to\infty} (D^{\pm}{\mathfrak p}_{\phi,\psi})(t)$ exist, are equal, and finite. We call their common value $\mathfrak p_{\phi,\psi}'(+\infty)$. Similarly,
$
\displaystyle\lim_{t\to-\infty} (D^{\pm}{\mathfrak p}_{\phi,\psi})(t)$ are equal and finite. We call their common value $\mathfrak p_{\phi,\psi}'(-\infty)$.
\item $\p_{\phi,\psi}'(+\infty)=\sup\{\int\psi d\mu:\mu\in\mathfs M(\Sigma^+)\}$, $\p_{\phi,\psi}'(-\infty)=\inf\{\int\psi d\mu:\mu\in\mathfs M(\Sigma^+)\}$.
\item
 If $\psi$ is not cohomologous to a constant by a continuous transfer function, then ${\mathfrak p}'_{\phi,\psi}(-\infty)<{\mathfrak p}'_{\phi,\psi}(+\infty)$.
\end{enumerate}
\end{theorem}
For subshifts of finite type (or countable Markov shifts with the BIP property), these results are proved in \cite[\S3]{Babillot-Ledrappier-Lalley}, \cite{Morris-Zero-Temp}. In this case more is true:  $\p_{\phi,\psi}(t)$ is real-analytic on $\R$, and the one-sided derivatives can be replaced by ordinary derivatives.
For general countable Markov shifts we do not necessarily have differentiability on $\R$, but the theorem is saved by  standard convexity arguments:

\begin{proof}
Fix some $\nu\in\mathfs M(\Sigma^+)$ carried by a periodic orbit, then
$\p_{\phi,\psi}(t)\geq P_{\nu}(\phi+t\psi)>-\infty$, proving that $\p_{\phi,\psi}(t)\neq -\infty$ on $\R$.  Next,
$\p_{\phi,\psi}(t)<\infty$, because for every $\mu\in\mathfs M(\Sigma^+)$,
$
P_\mu(\phi+t\psi)=h_\mu(\sigma)+\int\phi d\mu+t\int\psi d\mu\leq h+\sup\phi+|t|\|\psi\|_\infty<\infty.$ Passing to the supremum over $\mu\in \mathfs M(\Sigma^+)$ gives
$\p_{\phi,\psi}(t)\leq h+\sup\phi+|t|\|\psi\|_\infty<\infty$. So $\p_{\phi,\psi}$ is finite on $\R$.

To see convexity, we use the identity $P_G(\phi+t\psi)=\lim\limits_{n\to\infty}\frac{1}{n}\log Z_n(\phi+t\psi,a)$.
By H\"older's inequality,  if $t=s t_1+(1-s)t_2$ with $s\in [0,1]$, then
$$
Z_n(\phi+t\psi,a)\leq Z_n(\phi+t_1\psi,a)^{s}Z_n(\phi+t_2\psi,a)^{1-s}.
$$
It follows that  $t\mapsto \frac{1}{n}\log Z_n(\phi+t\psi,a)$ is convex on $\R$. Pointwise limits of convex functions are convex, therefore $\p_{\phi,\psi}(t)$ is convex on $\R$. We proved (1), (2).

 To see (3) we use convexity to note that $t\mapsto (D^{\pm}\mathfrak p_{\phi,\psi})(t)$ is increasing (in the broad sense), and  for all $t_1<t_2$
$$
(D^-\mathfrak p_{\phi,\psi})(t_1)\leq
(D^+\mathfrak p_{\phi,\psi})(t_1)\leq
(D^-\mathfrak p_{\psi,\phi})(t_2)\leq
(D^+\mathfrak p_{\psi,\phi})(t_2).
$$
This implies the existence and equality of the limits which define $\p_{\phi,\psi}'(\pm\infty)$.
To see that these quantities are finite, we note that $P_G(\phi)-|t|\|\psi\|_\infty\leq \p_{\phi,\psi}(t)\leq P_G(\phi)+|t|\|\psi\|_\infty$,
so $\p_{\phi,\psi}(t)$ is a convex function with  asymptotes of finite slope.

We prove (4). Let $p(t):=\mathfrak p_{\phi,\psi}(t)$, $p'_{\pm}:=D^{\pm}p$, $p'(\pm\infty):=\mathfrak p'_{\phi,\psi}(\pm\infty)$. By convexity,  for every $\mu\in\mathfs M(\Sigma^+)$ and $t>0$,
$$
p'(\infty)\geq p'_+(t)\geq \frac{p(t)-p(0)}{t}\geq \frac{(P_\mu(\phi)+t\int\psi d\mu)-p(0)}{t}\xrightarrow[t\to\infty]{}\int\psi d\mu.
$$
Thus $p'(\infty)\geq \sup\{\int\psi d\mu:\mu\in\mathfs M(\Sigma^+)\}$.

Next, we fix $\eps>0$ arbitrarily small and  $t>0$ so large that
$
p'(+\infty)\leq p'_-(t)+\eps.
$
This is possible by (3).
For every $0<\delta<\eps$,  we choose a $\sigma$-invariant $m:=m_{t,\delta}$ such that $p(t+\delta)\leq P_m(\phi)+(t+\delta)\int\psi dm+\delta^2$. By the variational principle, $p(t)\geq P_m(\phi)+t\int \psi dm$, so
\begin{align*}
p'(\infty)&\leq p'_-(t)+\eps\leq \frac{p(t+\delta)-p(t)}{\delta}+\eps\\
&\leq\frac{[P_m(\phi)+(t+\delta)\int\psi dm+\delta^2]-[P_m(\phi)+t\int\psi dm]}{\delta}+\eps\\
&=\int\psi dm+\delta+\eps\leq\sup\{\int \psi d\mu\}+2\eps.
\end{align*}
The $\sup$ runs over all $\sigma$-invariant probabilities.
Passing to the limit $\eps\to 0$, we obtain $p'(+\infty)\leq \sup\{\int\psi d\mu:\mu\in\mathfs M(\Sigma^+)\}$.
This shows (4) for $p'(+\infty)$.
The claim for $p'(-\infty)$ follows immediately using the identity $\mathfrak p_{\phi,-\psi}(t)=\mathfrak p_{\phi,\psi}(-t)$.

To see (5), assume the contrary: $p'(-\infty)=p'(\infty)=c$.
Then (4) tells us that  $\int\psi d\mu=c$ for all $\mu\in\mathfs M(\Sigma^+)$.
Applying this to invariant measures carried by a periodic orbit we find that
$
\sigma^n(\un{x})=\un{x}\Rightarrow \sum_{k=0}^{n-1}\psi(\sigma^k \un{x})=cn.
$
By the Livshits theorem, this implies that $\psi$ is cohomologous to $c$ via a continuous transfer function.
(The proof of the Livshits Theorem for subshifts of finite type given in \cite[Theorem~1.28]{Bowen-LNM}  works verbatim in the countable alphabet case.)
\end{proof}
\noindent
{\em Remark.\/} The assumption that the Gurevich entropy $h$ is finite is used in two places: In (1) we use it to show that $P_G(\phi+t\psi)<\infty$ for all $t$, and in (4) we use it implicitly to make sure that $P_m(\phi)$ and  $P_\mu(\phi)$ are well-defined. If we relax the assumption that $h<\infty$ to the assumption that $P_G(\phi)<\infty$, then the theorem remains true, except that the suprema in (4) should be taken over $\mathfs M_\phi(\Sigma^+)$, instead of $\mathfs M(\Sigma^+)$.

\begin{theorem}\label{t.TDF-for-SPR}
Suppose $\Sigma^+$ is a topologically mixing countable Markov shift with finite Gurevich entropy. Let $\phi$ be an SPR $\theta$-weakly H\"older continuous potential such that $\sup\phi<\infty$.
Then  there exist $M:=M_\theta(\phi)>1$ and $\eps:=\eps_\theta(\phi)>0$ such that for every $\beta>|\log\theta|$ and $\psi$ such that  $\|\psi\|_\beta\leq 1$, the following holds:
\begin{enumerate}[(1)]
\item ${\mathfrak p}_{\phi,\psi}(t)$ is  real-analytic on $(-\eps,\eps)$.
\item If $|t|\leq \eps$ then there exists a unique equilibrium measure $m_t:=\mu_{\phi+t\psi}$ for $\phi+t\psi$, i.e.\
    $
    \p_{\phi,\psi}(t)=P_{m_t}(\phi+t\psi)
    $, and
    $
    \p_{\phi,\psi}(t)>P_{\mu}(\phi+t\psi)$  for $\mu\neq m_t$.
\item If $|t|<\eps$, then ${\mathfrak p}_{\phi,\psi}'(t)=\E_{m_t}(\psi):=\int\psi dm_t.
$
\item If $|t|<\eps$, then  $\displaystyle{\mathfrak p}_{\phi,\psi}''(t)=\sigma_{m_t}^2(\psi)$
where
$\sigma_{m_t}^2(\psi):=\lim\limits_{n\to\infty}\frac{1}{n}\E_{m_t}[(\psi_n
    -\E_{m_t}(\psi_n))^2]$ and $\psi_n:=\sum_{k=0}^{n-1}\psi\circ\sigma^k$.
\item If $|t|<\eps$, then $|\p_{\phi,\psi}'(t)|, |\p_{\phi,\psi}''(t)|, |\p_{\phi,\psi}'''(t)|\leq M$.

\item If $\vf\in\mathcal H_\beta$ such that $\|\vf\|_\beta\leq 1$, then $(t,s)\mapsto P_G(\phi+t\psi+s\vf)$ has  continuous partial derivatives of all orders on  $(-\eps,\eps)^2$.
\end{enumerate}
\end{theorem}
\noindent
The proof of the theorem is long, and is deferred to \S\ref{s.TDF-for-SPR}.

\medskip
\noindent
{\em Remark 1.\/}
Theorem \ref{t.TDF-for-SPR}  is false without the  SPR assumption, because of ``phase transitions", see \cite{Sarig-CMP-2001}, \cite{Sarig-CMP-2006}. Some version of the theorem is also true for TMS with infinite Gurevich entropy, but in this case the measures $m_t$ should be taken to be the RPF measures of $\phi+t\psi$ and not their  equilibrium measures.

\medskip
\noindent
{\em Remark 2.\/}
 A weaker version of the theorem with $\eps,M$ depending also on  $\psi$  is  known.  For the finite alphabet case, see  \cite[Chapter 4]{Parry-Pollicott-Asterisque}, \cite{Guivarch-Hardy}, \cite{Ruelle-TDF-book}. For infinite alphabets, see  \cite{Cyr-Sarig}.
 We need $\eps,M$ to be independent of $\psi$ to describe the regime when the ``sharp" EKP inequality \eqref{sharp-EKP} holds,  see the end of Section \ref{sec:large_pressure}.

\medskip
We make a few more comments on the quantity $\sigma^2_{m_t}(\psi)$ in part (4) of the theorem. In general,  the {\em variance} of $\psi\in L^2(m)$ is
$
\mathrm{Var}_m(\psi):=\int [(\psi-\smallint\psi dm)^2]dm.
$
If $m$ is $\sigma$-invariant, then $\psi\in L^2(m)\Rightarrow\psi_n:=\sum_{k=0}^{n-1}\psi\circ\sigma^k\in L^2(m)$ for all $n$, and  the {\em asymptotic variance} of $\psi$ is the following limit whenever it exists:
\begin{equation}\label{e.asymp-var}
\sigma^2_{m}(\psi):=\lim_{n\to\infty}\frac{1}{n}\mathrm{Var}_m(\psi_n).
\end{equation}
The asymptotic variance is important because it is the variance of the Gaussian  distributional limit  of  the $m$-distributions of $(\psi_n-n\int\psi dm)/\sqrt{n}$, whenever $m$ is the equilibrium measure of an SPR weakly H\"older continuous potential $\phi$, and $\psi\in\mathcal H_\beta$. See  \cite[Chapter 4]{Parry-Pollicott-Asterisque} and \cite{Cyr-Sarig}.

The following theorem gives some of the properties of the asymptotic variance.
\begin{theorem}\label{t.sigma}
Let $\Sigma^+$ be a topologically mixing countable Markov shift with finite Gurevich entropy. Let $\phi$ be a $\theta$-weakly H\"older continuous SPR potential such that $\sup\phi<\infty$. Let $m$ be the unique equilibrium measure of $\phi$.
\begin{enumerate}[(1)]
\item For every $\psi\in\mathcal H_\beta$, $\sigma_{m}(\psi)=0$ if and only if $\psi=r-r\circ\sigma+c$ where $c\in\R$, and  $r$ is continuous (but perhaps not bounded).
\item If $\psi,\vf\in\mathcal H_\beta$ and $\psi-\vf=r-r\circ\sigma+c$ with $r$ continuous and $c\in\R$, then $\sigma_{m}(\psi)=\sigma_{m}(\vf)$.
\item Let $M:=M_\theta(\phi)$ be as in Theorem~\ref{t.TDF-for-SPR}, then $\sigma_{m}(\psi)\leq M \|\psi\|_\beta$  whenever $e^{-\beta}\leq \theta$ and  $\psi\in\mathcal H_\beta$.
\end{enumerate}
\end{theorem}
\noindent
The theorem is well-known for subshifts of finite type  \cite[Chapter 4]{Parry-Pollicott-Asterisque}, \cite{Guivarch-Hardy}, and follows from results in \cite{Cyr-Sarig} in the infinite alphabet case. See Section~\ref{s.TDF-for-SPR}.

\section{The  Proofs of Theorems~\ref{t.TDF-for-SPR} and~\ref{t.sigma}}\label{s.TDF-for-SPR}

The material in this section is not used in other parts of the paper and can be skipped at first reading.

As in \cite{Ruelle-TDF-book}, \cite{Parry-Pollicott-Asterisque}, \cite{Guivarch-Hardy}, the proof uses the {\em transfer operator method}: First we  represent  $\p_{\phi,\psi}(t)$ as the logarithm of the  leading eigenvalue $\lambda(t)$ of Ruelle's
 operator  $L_{\phi+t\psi}$ (see \eqref{eq:Ruelle-Operator}).
  Then we will analyze the dependence of $\lambda(t)$ on $t$, using perturbation theory. The perturbative analysis hinges on  the following fact from \S\ref{s.SGP}: If $\phi$ is SPR, then  $L_\phi$ acts with spectral gap on some ``nice" Banach space.

\subsection*{Standing assumptions for this section:}
Throughout this section, we suppose that $\Sigma^+$ is a topologically mixing TMS with positive finite Gurevich entropy, $\phi$ is an SPR $\theta$-weakly H\"older continuous potential such that $\sup\phi<\infty$, and $\psi\in\mathcal H_\beta$ where $e^{-\beta}\leq \theta$ and $\|\psi\|_\beta\leq 1$.
Since $\beta\mapsto\|\psi\|_\beta$ is monotonically increasing,  if  Theorem \ref{t.TDF-for-SPR} holds for the $\beta$ such that  $e^{-\beta}=\theta$, then it holds with the same $\eps,M$ for all $\beta$ such that $e^{-\beta}\leq \theta$. Henceforth we assume
$
e^{-\beta}=\theta.
$
 The next lemma does not require $\phi$ to be SPR.

\begin{lemma}\label{l.trivial}
For every state $a$, $\un{x}\in \Sigma^+$, $t\in\R$ and $n\geq 1$,  $(L_{\phi+t\psi}^n 1_{[a]})(\un{x})<\infty$ and
$
\p_{\phi,\psi}(t)=P_G(\phi+t\psi)=\lim\limits_{n\to\infty}\frac{1}{n}\log (L_{\phi+t\psi}^n 1_{[a]})(\un{x}).
$
\end{lemma}
\begin{proof}
Fix $t$, and let $\vf:=\phi+t\psi$.
Set $B:=\sum_{n\geq 2}\var_n(\vf)$.  Then  for every cylinder $[a,\xi_1,\ldots,\xi_{n-1},b]$ of length $n+1$ with $n\geq 1$, and for every $\un{x},\un{y}\in [a,\xi_1,\ldots,\xi_{n-1},b]$,
$$
|\vf_n(\un{x})-\vf_n(\un{y})|\leq B.
$$

Let $\vf_n^{\pm}[a,\xi_1,\ldots,\xi_{n-1},a]$ denote the supremum ($+$) or infimum ($-$) of $\vf_n$ on $[a,\xi_1,\ldots,\xi_{n-1},a]$. These differ from each other by at most $B$. It follows that
$$
e^{-B} \sum_{[a,\xi_1,\ldots,\xi_{n-1},a]\neq \emptyset}\!\!\!\!e^{\vf_n^+[a,\xi_1,\ldots,\xi_{n-1},a]}
\leq Z_n(\vf,a)\leq  e^B \sum_{[a,\xi_1,\ldots,\xi_{n-1},a]\neq \emptyset}\!\!\!\! e^{\vf_n^-[a,\xi_1,\ldots,\xi_{n-1},a]}.
$$
Similarly, one shows that if $[a,\un{\xi},a]$ and $[a,\un{\eta},a]$ are non-empty cylinders of lengths $m+1$ and $n+1$, then
$
|\vf_{m+n}^+[a,\un{\xi},a,\un{\eta},a]-\vf_{m}^-[a,\un{\xi},a]-\vf_{n}^-[a,\un{\eta},a]|\leq 2B.
$

These estimates can be used to show that
$
Z_n(\vf,a)^k\leq \exp(2kB) Z_{kn}(\vf,a).
$
By the standing assumptions,
$
\lim \frac{1}{\ell}\log Z_{\ell}(\vf,a)= P_G(\vf)\leq P_G(0)+\sup\vf=h+\sup\vf<\infty$, and therefore $Z_{kn}(\vf,a)<\infty$ for every $n \geq 1$, for all $k$ large enough. So
\begin{equation}\label{Z-n-finite}
Z_n(\vf,a)<\infty\text{ for all $n$.}
\end{equation}
If $\un{x}\in [a]$,
\[
\displaystyle(L_\vf^n 1_{[a]})(\un{x})=\sum_{\sigma^n(\un{y})=\un{x}}e^{\vf_n(\un{y})}1_{[a]}(\un{y})=
\sum_{[a,\xi_1,\ldots,\xi_{n-1},a]\neq \emptyset}
e^{\vf_n(a,\un{\xi},a,x_1,x_2,\ldots)}.
\]
The exponent is sandwiched between $\vf_n^\pm[a,\un{\xi},a]$, so by the previous paragraph,
\begin{equation}\label{e.L-Z}
e^{-B}Z_n(\vf,a)\leq (L_\vf^n 1_{[a]})(\un{x})\leq e^B Z_n(\vf,a)
\end{equation} for all $n$. It follows that  $L_\vf^n 1_{[a]}$ is finite on $[a]$ for all $n$, and $\frac{1}{n}\log (L_\vf^n 1_{[a]})(\un{x})\xrightarrow[n\to\infty]{}P_G(\vf)$ for all $\un{x}\in [a]$. This proves the lemma in the special case when $\un{x}\in [a]$.

Suppose $\un{x}\in [b]$ where $b\neq a$.
By topological mixing, there is a finite admissible path $b\un{\xi}a =(b,\xi_1,\ldots,\xi_{p-1},a)$.  If $\un{y}$ is a $\sigma^n$-pre-image of $\un{x}$ in $[a]$, then  $(b,\un{\xi},\un{y})$ is a $\sigma^{n+p}$-preimage  of $\un{x}$ in $[b]$.
Therefore
$$
e^{\vf_p^-[b\un{\xi} a]}(L_\vf^n 1_{[a]})(\un{x})=
e^{\vf_p^-[b\un{\xi} a]}\sum_{\sigma^n(\un{y})=\un{x}}e^{\vf_n(\un{y})}1_{[a]}(\un{y})\leq e^B (L_\vf^{n+p} 1_{[b]})(\un{x}).
$$
Since $\un{x}\in [b]$,
$
(L_\vf^{n+p} 1_{[b]})(\un{x})\leq \const Z_{n+p}(\vf,b)\text{ on }[b],$
by \eqref{e.L-Z}. It follows that
$$
L_\vf^n 1_{[a]}\leq \const Z_{n+p}(\vf,b)<\infty\text{ on }[b],
$$ and   $\limsup \frac{1}{n}\log L_\phi^n 1_{[a]}\leq P_G(\vf)$ on $[b]$.

Topological mixing also gives us a finite admissible path $(a,\eta_1,\ldots,\eta_{q-1},b)$. Let $\un{z}:=(a,\un{\eta},\un{x})$. If $\sigma^n(\un{y})=\un{z}$, then $\sigma^{n+q}(\un{y})=\un{x}$. So
\begin{align*}
&e^{-\vf_q(\un{z})}
(L_{\vf}^{n+q}1_{[a]})(\un{x})\geq (L_\vf^n 1_{[a]})(\un{z})\geq e^{-B}Z_n(\vf,a),
\end{align*}
and $\liminf \frac{1}{n}\log L_\phi^n 1_{[a]}\geq P_G(\vf)$ on $[b]$.
\end{proof}

The following lemma generalizes \cite[Theorem 3]{Sarig-CMP-2001} by making $\eps$ independent of $\psi$ (the importance of this is explained at the end of \S\ref{sec:large_pressure}). The weak H\"older continuity and the SPR property of $\phi$ are essential.

\begin{lemma}\label{l.perturb-thry}
There exists $\eps:=\eps_\theta(\phi)>0$ such that for every $\psi\in\mathcal H_\beta$ such that $\|\psi\|_\beta\leq 1$,
 $\p_{\phi,\psi}(t)$ is real-analytic on $(-\eps,\eps)$.
\end{lemma}

\begin{proof}
By our assumptions $\phi$ is $\theta$-weakly H\"older continuous, and strongly positively recurrent. As explained  in \S\ref{s.SGP}, this implies that $L_\phi$ acts with spectral gap on  a Banach space $(\mathcal L,\|\cdot\|_{\mathcal L})$ over $\C$, with properties (a)--(g) as listed there.
The space $\mathcal L$ depends on $\phi$ and on  $\theta$.\footnote{The dependence on $\theta$ is genuine: For some $\phi$, e.g. $\phi\equiv 0$, there is no ``canonical" $\theta$.}

Property (e) says that  the spectrum of  $L_\phi:\mathcal L\to \mathcal L$ consists of a simple eigenvalue $\lambda:=\exp P_G(\phi)$, and a compact set  $K_0\subset \{z\in\C:|z|<\lambda\}$. It is well-known that
this spectral picture persists for small perturbations of $L_\phi$.

Specifically,
let $\gamma'$ be a smooth parametrization of a circle in $\C$, with center zero, radius $R'<\lambda$, and such that $\gamma'$ contains $K_0$ in its interior.
Let $\gamma$ be a smoothly parameterized closed circle with center $\lambda$, and radius $R$ so small that $\gamma$ is completely outside $\gamma'$. By the theory of analytic perturbations of linear operators on Banach spaces, there exists $\delta_0:=\delta_0(\phi,\theta)>0$ as follows. If $L$ is a bounded linear operator such that $\|L-L_\phi\|<\delta_0$, then
\begin{enumerate}[(A)]
\item   the spectrum of $L$ consists of a simple eigenvalue $\lambda(L)$ inside $\gamma$,  and a compact subset $K_L$ inside $\gamma'$.
\item
    $P=P(L):=\frac{1}{2\pi i}\oint_\gamma (\xi I-L)^{-1}d\xi$
    is a well-defined operator such that $P^2=P$, $PL=LP=\lambda(L)P$, $\dim\mathrm{ Image}(P)=1$, and
    $$
    \mathrm{spectrum}(LP)=\{\lambda(L)\}\ , \ \mathrm{spectrum}(L(I-P))=K_L.
    $$
\end{enumerate}
See \cite[Theorem~IV-3.16]{Kato-Book}.

Define the operators
$\displaystyle (L_{z}f)(\un{x}):=\sum_{\sigma(\un{y})=\un{x}}e^{\phi(\un{y})+z\psi(\un{y})}f(\un{y})$ for $z\in\C$ and
$
M_\psi f= \psi f
$. By property (g), $M_\psi$ is a bounded linear operator on $\mathcal L$, and  $\|M_\psi\|\leq \|\psi\|_\beta\leq 1$. Thus
\begin{align}\label{Taylor-L}
&L_z=L_\phi M_{\exp z\psi}=\sum_{n=0}^\infty \frac{z^n}{n!}L_\phi M_{\psi}^n,
\end{align}
and this series converges absolutely  on $\mathcal L$ in the operator norm, because $\|L_\phi M_\psi^n\|\leq \|L_\phi\|\|\psi\|_\beta^n\leq \|L_\phi\|$. It follows that $L_z$ is a bounded linear operator on $\mathcal L$, and
$$
\|L_z-L_\phi\|\leq |z|\sum_{n=1}^\infty \frac{|z|^{n-1}}{n!}\|L_\phi\|\leq 2\|L_\phi\|\cdot |z|, \text{ for all }|z|\leq 1.
$$
Let $\eps^{(1)}:=\eps^{(1)}(\theta,\phi):=\min\{1,\frac{\delta_0}{2\|L_\phi\|}\}$. If $|z|<\eps^{(1)}$, then $\|L_z-L_\phi\|<\delta_0$, whence
$$
P_z:=P(L_z)
$$
is well-defined. Notice that $\eps^{(1)}$ is independent of $\psi$.

Next we claim that  there exist $\eps^{(2)}:=\eps^{(2)}(\theta,\phi)$ and $K:=K(\theta,\phi)$ such that for all $|z|<\eps^{(2)}$ and $\xi\in\gamma$, $\xi I-L_z$ is invertible, and
$
\|(\xi I-L_z)^{-1}\|\leq K.
$

To see this, recall that by choice of $\gamma$, the spectrum of $L_\phi$ does not intersect $\gamma$, and therefore $\xi I-L_\phi$ has a bounded inverse for all $\xi\in\gamma$. The resolvent set of $L_\phi$ is open,  and $\xi\mapsto (\xi I-L_\phi)^{-1}$ is analytic on it, therefore the norm of this inverse is locally bounded on $\gamma$, whence  by compactness, less than some global constant $K'=K'(\theta,\phi)$ everywhere on $\gamma$. The formal calculation
$$
(\xi I-L_z)^{-1}=(\xi I-L_\phi-(L_z-L_\phi))^{-1}=(\xi I-L_\phi)^{-1}\sum_{n=0}^\infty [(L_z-L_\phi)(\xi I-L_\phi)^{-1}]^n
$$
 shows that if $K'\|L_z-L_\phi\|\leq \frac{1}{2}$, then $\xi I-L_z$ is invertible, and the norm of the inverse is bounded by $2K'$.

 Let
$
\eps^{(2)}:=\eps^{(2)}(\theta,\phi):= \min\{\eps^{(1)},\frac{1}{4\|L_\phi\|K'}\}
$, $K:=K(\theta,\phi):=2 K'$. These constants are independent of $\psi$, and if $|z|<\eps^{(2)}$, then
$
\|(\xi I-L_z)^{-1}\|\leq K
$.

\medskip
We now investigate the properties of
$
P_z:=P(L_z)
$.

\medskip
$P_0$ is the eigenprojection of $\lambda_0:=\lambda$. By the spectral gap property for $\phi$ and  \cite[Lemma 8.1]{Cyr-Sarig},
\begin{equation}
	\label{eq:GRPF}
	P_0 f=h_0\int f d\nu_0
\end{equation}
where $h_0$ is a positive function, uniformly bounded away from zero and infinity on partition sets, and $\nu_0$ is a positive measure which is finite and positive on cylinders. In fact, by \eqref{eq:GRPF} $h_0 d\nu_0$ is the RPF measure, whence the equilibrium measure, of $\phi$.

\medskip
\noindent
{\sc Claim:} {\em $z\mapsto P_z$ is analytic on $\{z\in\C:|z|<\eps^{(2)}\}$.}

\medskip
\noindent
{\em Proof of the claim:\/} If $|z|,|w|<\eps^{(2)}$, then $P_z,P_w$ are well-defined, and by (B), $
\frac{P_z-P_w}{z-w}=\frac{1}{2\pi i}\oint_\gamma \frac{(\xi I-L_z)^{-1}-(\xi I-L_w)^{-1}}{z-w} d\xi
$.
The integrand satisfies the identity
$$
\frac{(\xi I-L_z)^{-1}-(\xi I-L_w)^{-1}}{z-w}=(\xi I-L_z)^{-1}\left(\frac{L_z-L_w}{z-w}\right)(\xi I-L_w)^{-1}.
$$
To see this, start with the left-hand side, pull $(\xi I-L_z)^{-1}$ to the left, and pull $(\xi I-L_w)^{-1}$ to the right.

It is not difficult to verify, using \eqref{Taylor-L}, that for every $|z|<\eps^{(2)}$,
$$
(\xi I-L_z)^{-1}\left(\frac{L_z-L_w}{z-w}\right)(\xi I-L_w)^{-1}\xrightarrow[w\to z]{}(\xi I-L_z)^{-1} L_z' (\xi I-L_z)^{-1},
$$
where
$L_z':=\sum_{n=1}^\infty \frac{z^{n-1}}{(n-1)!}L_\phi M_\psi^n$.

If  $|z|,|w|<\eps^{(2)}$ and $\xi\in\gamma$, then $\|(\xi I-L_z)^{-1}\|, \|(\xi I-L_w)^{-1}\|<K$, and
\begin{align*}
\left\|\frac{L_z-L_w}{z-w}\right\|&\leq \sum_{n=1}^\infty\frac{|z^n-w^n|}{n!|z-w|}\|L_\phi\|\|\psi\|_\beta^n\ \ \text{, by \eqref{Taylor-L},}\\
&\leq
\|L_\phi\|\sum_{n=1}^\infty \frac{|z|^{n-1}+|z|^{n-2}|w|+\cdots +|w|^{n-1}}{n!}< 3\|L_\phi\|.
\end{align*}
Therefore $\|(\xi I-L_z)^{-1}\left(\frac{L_z-L_w}{z-w}\right)(\xi I-L_w)^{-1}\|\leq 3K^2\|L_\phi\|$.

Now fix some arbitrary bounded linear functional $\vf$ on the Banach space of bounded linear operators on $\mathcal L$. By the previous discussion and the bounded convergence theorem,
$$
\lim_{w\to z}\vf\left(\frac{P_z-P_w}{z-w}\right)=\frac{1}{2\pi i} \oint_\gamma \vf\left[(\xi I-L_z)^{-1} L_z' (\xi I-L_z)^{-1} \right] d\xi.
$$
Thus $P_z$ is differentiable in the weak sense at  $z$.
By a well-known consequence of the Banach-Steinhaus theorem, this implies differentiability in the strong sense, and indeed holomorphy on $\{z\in\C:|z|<\eps^{(2)}\}$. The claim is proved.

\medskip
We note for future reference a by-product of the previous proof: If $|z|,|w|<\eps^{(2)}$, then
$$
\left\|\frac{P_z-P_w}{z-w}\right\|\leq \left\|\oint_\gamma
\frac{(\xi I-L_z)^{-1}-(\xi I-L_w)^{-1}}{z-w}d\xi
\right\|\leq \frac{3K^2\|L_\phi\|}{2\pi}\cdot\text{length}(\gamma).
$$
So $\|P_z-P_w\|\leq C|z-w|$ with $C:=C(\theta,\phi)$ independent of $\psi$.

\medskip
Our next task is to prove the analyticity of $\lambda_z:=\lambda(L_z)$ on some neighborhood of zero.
Pick a bounded linear functional $\vf$ on $\mathcal L$ such that $\vf(h_0)\neq 0$. Since
$
\|P_z-P_w\|\leq C|z-w|
$,
there exists  $0<\eps^{(3)}(\theta,\phi)<\eps^{(2)}$ independent of $\psi$ such that
$
\vf(P_z h_0)\neq 0 \text{ for all }|z|<\eps^{(3)}.
$
From (B) and linearity of $\vf$, we get that
$$
\displaystyle\lambda(L_z)=\frac{\vf(L_z P_z h_0)}{\vf(P_z h_0)},
$$
a ratio of two holomorphic non-vanishing functions on  $\{z:|z|<\eps^{(3)}\}$. So $\lambda_z$ is holomorphic there on $\{z:|z|<\eps^{(3)}\}$.

\medskip
We are now ready to prove the  lemma.
By \eqref{eq:GRPF}, $P_0 1_{[a]}=\nu_0[a] h_0\neq 0$, whence $\|P_0 1_{[a]}\|\neq 0$. Using $\|P_z-P_w\|\leq C|z-w|$, we find that for some $0<\eps^{(4)}(\theta,\phi)<\eps_\beta^{(3)}$ independent of $\psi$, $\|P_z 1_{[a]}\|\neq 0$ for all $|z|<\eps^{(4)}$.

Fix $|z|<\eps^{(4)}$ and choose some $\un{x}\in\Sigma^+$ such that $(P_z 1_{[a]})(\un{x})\neq 0$. Let $N_z:=L_z(I-P_z)$. Since $P_z^2=P_z$ and $P_z L_z=\lambda(L_z)P_z$,  we have  $P_z N_z=N_z P_z=0$. By (A) and (B),  the spectrum of $N_z$ is inside $\gamma'$ and $\lambda_z$ is outside $\gamma'$, therefore $\lambda_z^{-n}N_z^n\to 0$ in norm. So
$$
\lambda_z^{-n}L_z^n=\lambda_z^{-n}(\lambda_z P_z+N_z)^n=\lambda_z^{-n}(\lambda_z^n P_z+N_z^n)\to P_z\text{ in $\mathcal L$}.
$$
By property (c) of  $\mathcal L$,  $\lambda_z^{-n}(L_z^n 1_{[a]})(\un{x})\to (P_z 1_{[a]})(\un{x})\neq 0$ for all $\un{x}$.

We now specialize to the case of {\em real} $t$, $|t|<\eps^{(4)}$. For such $t$, $(L_t^n 1_{[a]})(\un{x})=(L_{\phi+t\psi}^n 1_{[a]})(\un{x})$ is real-valued. Necessarily, $\lambda_t$ and $(P_t 1_{[a]})(\un{x})$ are also real-valued.\footnote{If $z^n r_n\to w\neq 0$ and $r_n$ are real, then $z,w$ must also be real, otherwise $\frac{z^n r_n}{|z^n r_n|}\not\to \frac{w}{|w|}$. } Passing to natural logarithms, we find using Lemma~\ref{l.trivial} that
$$
\log\lambda_t=\lim_{n\to\infty}\frac{1}{n}\log (L_{\phi+t\psi}^n 1_{[a]})(\un{x})=\p_{\phi,\psi}(t)\text{ whenever }|t|<\eps^{(4)}.
$$
By (A), (B) and the choice of  $\eps^{(4)}$,  $z\mapsto \lambda_z$ maps $\{z:|z|<\eps^{(4)}\}$ holomorphically into the  interior of $\gamma$. Zero is outside $\gamma$. If  $\Log(z)$ is  a branch of the complex logarithm which  is holomorphic on the interior of $\gamma$, then
 $\Log\lambda_z$ is holomorphic on  $\{z\in\C:|z|<\eps^{(4)}\}$, and agrees with   $\p_{\phi,\psi}(t)$ on  $(-\eps^{(4)},\eps^{(4)})$.

 It follows that $\p_{\phi,\psi}(t)$ is real-analytic on $(-\eps_\theta(\phi),\eps_\theta(\phi))$, where $\eps_\theta(\phi):=\eps^{(4)}$.
\end{proof}

\begin{lemma}\label{l.SGP-for-t-psi}
There exists $\eps_\theta(\phi)>0$ such that for every $\psi\in\mathcal H_\beta$ for which $\|\psi\|_\beta\leq 1$,
 $\phi+t\psi$ has the spectral gap property for each $|t|<\eps_\theta(\phi)$.
\end{lemma}
\begin{proof}
We continue with the notation of the proof of the previous lemma. Take $\eps:=\eps^{(4)}(\theta,\phi)$ and $t$ real such that $|t|<\eps$. We claim that if $\|\psi\|_\beta\leq 1$, then $\phi+t\psi$ satisfies the spectral gap property with the Banach space $\mathcal L$ from the previous proof.

Property (a) for $\phi+t\psi$ follows from property (a) for $\phi$, because $\psi$ is bounded.
Properties (b), (c) and (g) are obvious, because they do not involve $t$. Property (d) is because thanks to (g),    the series \eqref{Taylor-L} converges in norm.

Property (e) is because
 for all  $|t|<\eps^{(1)}$,
  $\|L_t-L_\phi\|<\delta_0$,
 and therefore there is a projection $P_t$ onto a one-dimensional space such that $P_t L_t=L_t P_t$, $L_t P_t=\lambda_t P_t$,
 and so that the spectrum of $N_t:=L_t(I-P_t)$ is contained in a disc with center at the origin and radius strictly less than $|\lambda_t|$.
 In addition, as we saw at the end of the previous proof, if  $|t|<\eps^{(4)}$, then  $\lambda_t=\exp \p_{\phi,\psi}(t)=\exp P_G(\phi+t\psi)$.

 Property (f) is because for every $\psi'\in \mathcal H_\beta$,
 we can apply \eqref{Taylor-L} to the potential $\phi+t\psi$:
 $
L_{\phi+t\psi+z\psi'}=L_{\phi+t\psi}M_{\exp (z\psi')}=\sum_{n=0}^\infty \frac{z^n}{n!} L_{t} M_{\psi'}^n.
$
The series converges absolutely in operator norm by properties (f) and (g) for $\phi$.
In particular, it is bounded and analytic in $z$ on $\C$.
\end{proof}

\begin{lemma}\label{l.Cauchy-Bounds}
There exist $\eps_\theta(\phi),M_\theta(\phi)>0$  so that for every $\psi\in\mathcal H_\beta$ such that $\|\psi\|_\beta\leq 1$ and for every $|t|<\eps_\theta(\phi)$,  $|\p_{\phi,\psi}'(t)|, |\p_{\phi,\psi}''(t)|, |\p_{\phi,\psi}'''(t)|\leq M_\theta(\phi)$.
\end{lemma}
\begin{proof}
Let $\eps^{(4)}$ and $\delta_0$ be as in  the proof of Lemma~\ref{l.perturb-thry}. Then  $\p_{\phi,\psi}(t)$ extends to a holomorphic function $f(z)=\Log\lambda_z$ on a neighborhood of   $\{z:|z|<\eps^{(4)}\}$, where  $\lambda_z:=\lambda(P_z)$ is an eigenvalue of  operator $L_z$ such that $\|L_z-L_\phi\|<\delta_0$, and $\Log(z)$ is a suitable branch of the complex logarithm.

By the choice of $\delta_0$, $\lambda_z$ is in the interior of  some fixed circle $\gamma$ and outside some fixed circle $\gamma'$ surrounding zero, with $\gamma,\gamma'$ independent of $\psi$.
So $|f(z)|$ is uniformly bounded on $\{z:|z|<\eps^{(4)}\}$ by some constant which is independent of $\psi$.

Take $\eps:=\frac{1}{2}\eps^{(4)}$.
The lemma now follows from Cauchy's integral formula for the derivatives of $f(z)$ (``Cauchy's bounds").
 \end{proof}

\begin{lemma}\label{l.Hartogs}
There exists $\eps_\theta(\phi)>0$ such that for all $\psi,\vf\in\mathcal H_\beta$ for which $\|\psi\|_\beta,\|\vf\|_\beta\leq 1$, the function
$
p(s,t):=P_G(\phi+t\psi+s\vf)
$
has continuous partial derivatives of all orders on $(-\eps_\theta(\phi),\eps_\theta(\phi))^2$.
\end{lemma}
\begin{proof}
The proof is similar to the proof of  Lemma~\ref{l.perturb-thry}, so we only sketch it. Let $\mathcal L$ be the Banach space in that proof, and define  the operator $$L_{z,w} f:=L_{\phi+z\psi+w\vf}f\ \ \ \ (z,w\in\C).
$$
Then
$
L_{z,w}=L_\phi M_{\exp z\psi} M_{\exp w\vf}=\sum_{n,m=0}^\infty
\frac{z^n w^m}{n! m!}L_\phi M_{\psi}^n M_{\vf}^m,
$
and the series converges in norm for all $z,w\in\C$,  because
$\|L_{\phi} M_{\psi}^n M_{\vf}^m\|\leq \|L_{\phi}\|\|\psi\|_\beta^n\|\vf\|_\beta^m\leq \|L_\phi\|$ by property (g) of $\mathcal L$.
So $L_{z,w}$ are well-defined bounded linear operators on $\mathcal L$.

The series representation for $L_{z,w}$ implies that for all $|z|,|w|\leq 1$,$|z_0|,|w_0|\leq 1$,
\begin{align*}
&\|L_{z,w}-L_{z_0,w_0}\|\leq \|L_\phi\|{\sum_{n,m=0}^{\infty}}\frac{|z^n w^m-z^n_0 w^m_0|}{n! m!}\\
&\leq
\|L_\phi\|\underset{(n,m)\neq (0,0)}{\sum_{n,m=0}^{\infty}}\frac{|z|^n|w^m-w^m_0|+|w_0|^m|z^n-z_0^n|} {n! m!}\leq e^2\|L_\phi\|(|z-z_0|+|w-w_0|).
\end{align*}
In particular, if  $\delta_0$ is as in the proof of Lemma~\ref{l.perturb-thry}, then there exists $\kappa^{(1)}>0$ independent of $\psi,\vf$ such that  $\|L_{z,w}-L_\phi\|<\delta_0$ for all $|z|,|w|\leq \kappa^{(1)}$.

By the definition of $\delta_0$, for such $z,w$, the spectrum of $L_{z,w}$ does not intersect $\gamma$, therefore
$(\xi I-L_{z,w})^{-1}$ is well-defined and bounded for all $\xi\in\gamma$. Arguing as in the proof of Lemma~\ref{l.perturb-thry}, we find $0<\kappa^{(2)}<\kappa^{(1)}$ and $K>0$ independent of $\psi,\vf$ such that for all $|z|,|w|\leq \kappa^{(2)}$ and for all $\xi$ on $\gamma$,
$$
\|(\xi I-L_{z,w})^{-1}\|<K.
$$
As in the proof of that lemma, this can be used to show that $(z,w)\mapsto P_{z,w}$ is analytic separately in each of its variables on $\{(z,w)\in\C^2: |z|,|w|<\kappa_\beta^{(2)}\}$.

It follows that for every bounded linear functional $F$ on the Banach space of bounded linear operators on $\mathcal L$,
$
(z,w)\mapsto F(P_{z,w})
$
is holomorphic in $z$ and in $w$, whence
by Hartogs' theorem, in both variables on $\{(z,w):|z|,|w|<\kappa^{(2)}\}$.
In particular, $F(P_{z,w})$  has continuous partial derivatives of all orders there.

Since this holds for all bounded linear functionals $F$, $P_{z,w}$ has continuous partial derivatives of all orders on $\{(z,w):|z|,|w|<\kappa^{(2)}\}$.
We now continue exactly as in the proof of Lemma~\ref{l.perturb-thry} to construct $0<\kappa^{(4)}<\kappa^{(3)}<\kappa^{(2)}$  independent of $\psi,\vf$ and a branch of the complex logarithm $\Log(z)$ such that
\begin{enumerate}[(1)]
\item  $\Log\lambda(P_{z,w})$ has  partial derivatives of all orders on $\{(z,w):|z|,|w|<\kappa^{(3)}\}$
\item $\Log\lambda(P_{t,s})=P_G(\phi+t\psi+s\vf)
$
for all $t,s$ real such that $|t|,|s|<\kappa^{(4)}$.
\end{enumerate}
It follows that  $(t,s)\mapsto P_G(\phi+t\psi+s\vf)$ has continuous partial derivatives of all orders on $\{(t,s)\in\R^2: |t|,|s|<\kappa_\beta^{(4)}\}$.
\end{proof}

\begin{proof}[{\bf Proof of Theorem~\ref{t.TDF-for-SPR}}]
Let $\eps:=\eps(\theta,\phi)$ be the minimum of the epsilons in the previous lemmas, and let $M:=M_\theta(\phi)$ be as in Lemma \ref{l.Cauchy-Bounds}.

\medskip
\noindent
{\bf Part (1):} This is Lemma~\ref{l.perturb-thry}.

\medskip
\noindent
{\bf Part (2):}  By Lemma~\ref{l.SGP-for-t-psi}, $\phi+t\psi$ has the spectral gap property.  As explained in sections \ref{s.SGP} and \ref{s.eq}, this implies that $\phi$ is strongly positively recurrent, whence positively recurrent. Let $m_t$ denote the RPF measure of $\phi+t\psi$ (see \S~\ref{s.eq}).

$\Sigma^+$ is topologically mixing with  finite Gurevich entropy $h$. So $h_{m_t}(\sigma)\leq h<\infty$, and by Theorem \ref{t.eq-measure},  $m_t$ is the unique equilibrium measure of $\phi+t\psi$.

\medskip
\noindent
{\bf Part (3):} The proof of this and the following part is essentially in \cite{Guivarch-Hardy} and \cite[Chapter 4]{Parry-Pollicott-Asterisque}, but we give it for completeness.

Let $\lambda_z, L_z,P_z$ be as in the proof of Lemma \ref{l.perturb-thry}.
By part (2), if $t$ is real and  $|t|<\eps$, then
 $\phi+t\psi$ is positively recurrent, whence by the generalized Ruelle's Perron-Frobenius theorem, there is a positive continuous function $h_t$ and a measure $\nu_t$ finite and positive on cylinders such that
$$
L_t h_t=\lambda_t, L_t^\ast \nu_t=\lambda_t\nu_t, \int h_t d\nu_t=1.
$$
By \cite[Lem 8.1]{Cyr-Sarig}, $h_t\in\mathcal L$, and  $P_t f=h_t\int f d\nu_t$ for all $f\in\mathcal L$.

We saw in the proof of Lemma \ref{l.perturb-thry} that $t\mapsto \lambda_t, P_t, L_t$ are differentiable, and  $\lambda_t=\exp\p_{\phi,\psi}(t)$.
By \eqref{Taylor-L}, $L_t'f:=\frac{d}{dt}L_t f=L_t(\psi f)=L_tM_\psi f$.
Differentiating both sides of the identity $L_t P_t=\lambda_t P_t$ gives
\[
L_tM_\psi P_t+L_tP_t' = {\mathfrak p}_{\phi,\psi}'(t)\lambda_tP_t +\lambda_tP_t'.
\]
We  multiply by $P_t$ and cancel the equal  terms $P_t L_t P_t'=\lambda_t P_t P_t'$, with the result
\[
\lambda_tP_tM_\psi P_t= {\mathfrak p}_{\phi,\psi}'(t)\lambda_tP_t.
\]
We now apply the operators on the two sides of the equation to $h_t$ (which belongs to $\mathcal L$), and obtain
$
\lambda_t P_t(\psi h_t)=\p_{\phi,\psi}'(t)\lambda_t h_t
$. Substituting $t=0$ and noting that $\lambda_0=\exp P_G(\phi)\neq 0$, we obtain
$
\p_{\phi,\psi}'(0)=\int \psi h_0 d\nu_0.
$
Since $h_0\nu_0$ is the equilibrium measure of $\phi$, we are done.

\medskip
\noindent
{\bf Part (4):}
It is enough to consider the special case when $P_G(\phi)=0$ and $\int\psi d m_0=0$, since the general case can be reduced to this one by subtracting constants from $\phi$ and $\psi$. In this case $\lambda_0=1$ and $\lambda_0'=\frac{d}{dt}\big|_{t=0}\lambda_t=\int\psi dm_0=0$.

Fix $n$, differentiate $L^n_t P_t=\lambda_t^n P_t$ twice, then apply  $P_t$  from the left, and substitute $t=0$, dropping all the terms which contain $\lambda_0'$.  Collecting terms we obtain
\[
n{\mathfrak p}_\psi''(0)P_0=P_0M_{\psi_n}^2P_0+2P_0M_{\psi_n}P'_0.
\]
Applying this to $h_0/n$,
writing $H'_0=P'_0h_0$ and using $P_0f=h_0\nu_0(f)$,
we get
\begin{equation}\label{volt}
{\mathfrak p}_\psi''(0) =\frac{1}{n}m_0\left(\left(\psi_n\right)^2\right)+
2\nu_0\left(\frac{\psi_n}{n}H'_0\right).
\end{equation}
To complete the proof, it remains to show that $\nu_0\left(\frac{\psi_n}{n}H'_0\right)\to 0$.

{\em The integrand $\frac{\psi_n}{n}H_0'$ tends to zero $\nu_0$-a.e.:} First,  $\psi$ is bounded and $m_0$ is ergodic, so  $\frac{\psi_n}{n}\to \int \psi dm_0=0$ $m_0$-almost everywhere. Second,  $h_0>0$, so $\nu_0=\frac{1}{h_0}m_0\ll m_0$, and $\frac{\psi_n}{n}\to 0$ $\nu_0$-almost everywhere.

{\em The integrand $\frac{\psi_n}{n}H_0'$ is dominated by an $L^1(\nu_0)$--function:}  $h_0\in\mathcal L$, so $H'_0\in \mathcal L$, whence by property (b) of $\mathcal L$, $|H'_0|\in\mathcal L$. So $P_0 |H_0'|\in\mathcal L$. We saw above that $P_0 f=h_0 \nu_0(f)$ for all $f\in\mathcal L$. In particular, $\nu_0(|H_0'|)<\infty$. So  $|\frac{\psi_n}{n}H_0'|\leq \|\psi\|_\infty |H_0'|$ and $H_0'\in L^1(\nu_0)$. This is  the dominating $L^1$ function.

 By the dominated convergence theorem, $\nu_0\left(\frac{\psi_n}{n}H'_0\right)\to0$.

\medskip
\noindent
{\bf Parts (5) and (6):} This is the content of Lemmas \ref{l.Cauchy-Bounds} and \ref{l.Hartogs}.
\end{proof}

\begin{proof}[{\bf Proof of Theorem~\ref{t.sigma}}]\ \\
\noindent
{\bf Part (2):}
Let $\psi_n:=\psi+\psi\circ\sigma+\cdots+\psi\circ\sigma^{n-1}$.
Recall the definitions
\[
\p_{\phi,\psi}(t)=P_G(\phi+t\psi)=\lim_{n\to\infty}\frac{1}{n}\log Z_n(\phi+t\psi,a),
\]
\[\text{ where }Z_n(\phi+t\psi, a)=\sum_{\sigma^n(\un{x})=\un{x}} e^{\phi_n(\un{x})+t\psi_n(\un{x})}1_{[a]}(\un{x}).
\]
If $\psi=\vf+r-r\circ\sigma+c$, then $\psi_n(\un{x})=\vf_n(\un{x})+cn$ for all $x$ such that $\sigma^n(\un{x})=\un{x}$. Since the Gurevich pressure is defined in terms of periodic orbits,
$
\p_{\phi,\psi}(t)=\p_{\phi,\vf}(t)+ct
$. So $\p_{\phi,\psi}''=\p_{\phi,\vf}''$, and
by Theorem~\ref{t.TDF-for-SPR}, $\sigma^2_{m}(\psi)=\p_{\phi,\psi}''(0)=\p_{\phi,\vf}''(0)=\sigma^2_{m}(\vf)$.

\medskip
\noindent
{\bf Part (1):} Suppose $\psi=r-r\circ\sigma+c$ with $r$ continuous and $c$ a constant. Then $\sigma_m(\psi)=\sigma_m(0)=0$, by part 2.

Now suppose $\sigma_m(\psi)=0$. Without loss of generality, $\int\psi dm=0$ (otherwise work with $\psi-\int\psi dm$).
 By Lemma~\ref{l.SGP-for-t-psi}, $\phi$ has the spectral gap property, and we are in the situation discussed in \cite[Appendix B]{Cyr-Sarig}.
 The expression $\sigma^2$ defined there agrees with $\sigma^2_{m}(\psi)$ by \cite[Equation (8.4)]{Cyr-Sarig} and part (4) of Theorem~\ref{t.TDF-for-SPR}.\footnote{Note that  the $\lambda_t$ in \cite{Cyr-Sarig} is what we call in this paper $\lambda_{it}$.} Using the argument in  \cite[pp.\ 664-665]{Cyr-Sarig}, we find that $\psi$ must be a coboundary with a continuous transfer function.

\medskip
\noindent
{\bf Part (3):} Let $M$ be the constant from Theorem \ref{t.TDF-for-SPR}(5).
If $\psi\equiv 0$, then $\sigma_m(\psi)=0$ and there is nothing to prove. Suppose $\psi\not\equiv 0$. It is easy to check that $\sigma_{m}(t\psi)=|t|\sigma_{m}(\psi)$.
So
$
{\sigma_{m}(\psi)}={\|\psi\|_\beta}\sigma_{m}(\vf)$, where $\vf:=\frac{\psi}{\|\psi\|_\beta}$. Since $\|\vf\|_\beta=1$,
$
\sigma_m^2(\vf)=\p_{\phi,\vf}''(0)\leq M.
$
So $\sigma_m(\psi)\leq \|\psi\|_\beta\sigma_m(\vf)\leq M^{1/2}\|\psi\|_\beta\leq M\|\psi\|_\beta
$, where the last inequality is because $M>1$.
\end{proof}

\section{The restricted pressure function}

Suppose $\sigma:\Sigma^+\to\Sigma^+$ is a topologically mixing countable Markov shift with finite Gurevich entropy $h$, and $\phi$ is a weakly H\"older continuous function such that $\sup\phi<\infty$. Fix $\psi\in\mathcal H_\beta$ which is not cohomologous to a constant via a continuous transfer function. By Theorem~\ref{t.pressure-convex-prop} $\p_{\phi,\psi}'(-\infty)<\p_{\phi,\psi}'(+\infty)$, and we can make the following definition:

\begin{definition}
The {\em restricted pressure function of $\phi$, constrained on $\psi$,}
is  \newline
$\q_{\phi,\psi}:(\p_{\phi,\psi}'(-\infty),\p'_{\phi, \psi}(+\infty))\to\R$, given by
\[
\q_{\phi,\psi}(a):=\sup\{h_\mu(\sigma)+\int\phi d\mu: \mu\in\mathfs M(\Sigma^+), \int \psi d\mu=a\}.
\]
\end{definition}
\noindent

\begin{lemma}\label{l.h-is-concave}
Under the assumptions above,
$\q_{\phi,\psi}$ is a well-defined finite valued concave function, and $\q_{\phi,\psi}$  is bounded from above.
\end{lemma}
\begin{proof}
Since $h<\infty$ and  $\sup\phi<\infty$, $\mathfs M_\phi(\Sigma^+)=\mathfs M(\Sigma^+)$ and $h_\mu(\sigma)+\int\phi d\mu$ is well-defined for every $\mu\in\mathfs M(\Sigma^+)$.
By Theorem~\ref{t.pressure-convex-prop}, $a\in ({\mathfrak p}'_{\phi,\psi}(-\infty),{\mathfrak p}'_{\phi,\psi}(+\infty))$ if and only if  $\inf\{\int\psi d\mu\}<a<\sup\{\int\psi d\mu\}$,
therefore there exist two invariant measures $\mu_1,\mu_2$ such that $\int \psi d\mu_1<a<\int \psi d\mu_2$. Then $\int \psi d\mu=a$ for some convex combination $\mu$ of $\mu_1,\mu_2$, and the set
$\{\mu:\int\psi d\mu=a\}$  is non-empty. So the supremum in the definition of $\q_{\phi,\psi}(a)$ is over a non-empty set, and $\q_{\phi,\psi}(a)$ is well-defined.

Let's choose the measures  $\mu_1,\mu_2$ more carefully. A standard ergodic decomposition argument shows that we may choose $\mu_1,\mu_2$ to be ergodic. On countable Markov shifts,  $\mu_i$ generic orbits are limits of periodic orbits, therefore we may choose  $\mu_1,\mu_2$ to be   ergodic measures sitting on periodic orbits.  For such measures  $\int\phi d\mu_i>-\infty$, whence $\int \phi d\mu>-\infty$ for every convex combination of $\mu_1,\mu_2$. Looking at the argument in the previous section, we find that
$
\q_{\phi,\psi}(a)>-\infty.
$

This function is uniformly bounded above by $h+\sup\phi<\infty.
$
So $\q_{\phi,\psi}$ is finite on its domain, and bounded from above.

Concavity is because for all $a_1,a_2$ in the domain, for all $0\leq t\leq 1$, and for all $\eps>0$, if $\mu_1,\mu_2$ are invariant measures such that
$$
P_{\mu_i}(\phi)\geq \q_{\phi,\psi}(a_i)-\eps,\text{ and }\int \psi d\mu_i=a_i,
$$
then $\mu:=t\mu_1+(1-t)\mu_2$ is an invariant measure such that $\int\psi d\mu=ta_1+(1-t)a_2$ and hence by the affine properties of the Kolmogorov-Sinai entropy,
$
P_{\mu}(\phi)=tP_{\mu_1}(\phi)+(1-t) P_{\mu_2}(\phi).
$
So $\mathfrak{q}_{\phi,\psi}(ta_1+(1-t)a_2)\geq P_{\mu}(\phi)\geq t \mathfrak{q}_{\phi,\psi}(a_1)+(1-t)\mathfrak{q}_{\phi,\psi}(a_2)-\eps$. Since $\eps$ was arbitrary, the concavity of $\mathfrak{q}_{\phi,\psi}$ follows.
\end{proof}

The restricted pressure function  is understood very well for subshifts of finite type \cite[\S 3]{Babillot-Ledrappier-Lalley}, and for countable Markov shifts with the BIP property \cite{Morris-Zero-Temp}. In these cases, and if $\psi$ is not cohomologous to a constant, then  this function is smooth and strictly concave.

This is not true in general in the infinite alphabet case, because of the phenomenon of ``phase transitions" \cite{Sarig-CMP-2001},\cite{Sarig-CMP-2006}.
However, as the following theorem shows, in the SPR case  there is an explicit subinterval of the domain  where $\q_{\phi,\psi}$ is smooth and uniformly concave.
Crucially to the applications we have in mind, this interval  can be chosen in a way which  depends on $\psi$ only through $\E_{m}(\psi)$ and  $\sigma_{m}^2(\psi)$, where $m$ is the equilibrium measure of $\phi$.

\begin{lemma}\label{lemma-h}
Let $\Sigma^+$ be a topologically mixing countable Markov shift with finite Gurevich entropy. Let $\phi$ be a $\theta$-weakly H\"older continuous SPR potential such that $\sup\phi<\infty$.
Let $m$ denote the unique equilibrium measure of $\phi$, and suppose $e^{-\beta}\leq \theta$.
Then $\exists\delta_\theta(\phi),H_\theta(\phi)>0$ as follows.
For all $\psi\in\mathcal H_\beta$ such that $\|\psi\|_\beta\leq 1$, and $\sigma_{m}(\psi)\neq 0$:
\begin{enumerate}[(1)]
\item $
I_\psi:=\{t\in\R: |t-{\textstyle \int}\psi dm|<\delta_\theta(\phi)\sigma_{m}^4(\psi)\}
$
 is contained in the domain of $\mathfrak{q}_{\phi,\psi}$.
\item $\mathfrak{q}_{\phi,\psi}$ is uniformly bounded and  differentiable infinitely many times on $I_\psi$.

\item  If $a\in I_\psi$, then
    $
    -2\sigma_{m}^{-2}(\psi)\leq \mathfrak{q}_{\phi,\psi}''(a)\leq -\frac{1}{2}\sigma_{m}^{-2}(\psi),\ \ |\mathfrak{q}_{\phi,\psi}'''(a)|\leq H_\theta(\phi)\sigma_{m}^{-6}(\psi).
    $
In particular, $\mathfrak{q}_{\phi,\psi}$ is strictly concave on $I_\psi$.

\item If  $a_0:=\int\psi dm$, then
$
\mathfrak{q}_{\phi,\psi}(a_0)=P_G(\phi),\ \ \mathfrak{q}_{\phi,\psi}'(a_0)=0,\ \ \mathfrak{q}_{\phi,\psi}''(a_0)=-\frac{1}{\sigma_{m}^2(\psi)}.
$

\end{enumerate}
\end{lemma}

\begin{proof}
Fix $0<\beta\leq |\log\theta|$ and let $\eps:=\eps_\theta(\phi)$, $M:=M_\theta(\phi)$ as in  Theorems~\ref{t.TDF-for-SPR} and \ref{t.sigma}. Without loss of generality, $0<\eps<0.01$ and $M>100$.
Next, fix $\psi\in\mathcal H_\beta$ with norm $\|\psi\|_\beta\leq 1$ and such that
$\sigma_{m}(\psi)\neq 0$.
Let
$$
\sigma:=\sigma_{m}(\psi),\  a_0:=\int\psi dm,\   p(t):={\mathfrak p}_{\phi,\psi}(t),\
 q(a):=\mathfrak{q}_{\phi,\psi}(a).
$$

\medskip
\noindent
{\sc Claim.} {\em If $|a-a_0|<\frac{\eps \sigma^4}{2M^2}$, then  $\exists! t\in\R$ such that $p'(t)=a$. In addition, $\phi+t\psi$ has a unique equilibrium measure $m_{t}$, $q(a)=P_{m_t}(\phi)$, and

\begin{equation}\label{t-estimate}
|t|\leq \frac{\eps\sigma^2}{M^2}\ , \ (1-\eps)\sigma^2\leq p''(t)\leq (1+\eps)\sigma^2.
\end{equation}
}

\noindent
{\em Proof of the claim.\/}
By Theorems~\ref{t.TDF-for-SPR} and \ref{t.sigma}, $p'$ is real-analytic on $(-\eps,\eps)$,  $p'(0)=a_0$, $p''(0)=\sigma^2$, and $|p'''|\leq M$ on $(-\eps,\eps)$, and $\frac{\eps\sigma^2}{M^2}\leq \eps$.
By Taylor's approximation for $p'$, with the Lagrange form of the remainder,
\begin{align*}
p'(\tfrac{\eps\sigma^2}{M^2})&\geq p'(0)+p''(0)\tfrac{\eps\sigma^2}{M^2}-\tfrac{1}{2!}M \left(\tfrac{\eps\sigma^2}{M^2}\right)^2>a_0+\frac{\eps \sigma^4}{2M^2}\geq a.
\end{align*}
Similarly, $
p'(-\tfrac{\eps\sigma^2}{M^2})\leq p'(0)-p''(0)\tfrac{\eps\sigma^2}{M^2}+\frac{1}{2!}M \left(\tfrac{\eps\sigma^2}{M^2}\right)^2<a_0-\frac{\eps \sigma^4}{2M^2}\leq a.$
By the intermediate value theorem, $\exists t\in (-\tfrac{\eps\sigma^2}{M^2},\tfrac{\eps\sigma^2}{M^2})$ such that $p'(t)=a$.

To see that this $t$ is unique, we recall that $p$ is convex, therefore $p'$ is monotonically increasing in the broad sense, therefore by the previous inequalities, all solutions to $p'(t)=a$ must belong to $(-\tfrac{\eps\sigma^2}{M^2},\tfrac{\eps\sigma^2}{M^2})$. Inside  this interval, there can be at most one solution, because if there were $t_1\neq t_2$ such that $p'(t_i)=a$, then  $p''$ would have vanished somewhere in $(-\tfrac{\eps\sigma^2}{M^2},\tfrac{\eps\sigma^2}{M^2})$, whereas
$$
|p''(t)-\sigma^2|=|p''(t)-p''(0)|\leq M|t|\leq \eps\sigma^2\ \ \ (\because |p'''|\leq M, |t|\leq \tfrac{\eps\sigma^2}{M^2}),
$$
so $p''(t)>0$ on this  interval. Indeed, $(1-\eps)\sigma^2\leq p''(t)\leq (1+\eps)\sigma^2$ there.

Let $t$ be the unique solution to $p'(t)=a$.
We saw that $|t|\leq \frac{\eps\sigma^2}{M^2}$.
By Theorem~\ref{t.sigma}, $\sigma\leq M$,
so $|t|\leq \eps$. By Theorem~\ref{t.TDF-for-SPR} and by the choice of $\eps$, $\phi+t\psi$ has a unique equilibrium measure $m_t$.
Also, by the choice of $t$,
$
a=p'(t)=\int\psi dm_t
$.
 Thus
$$P_{m_t}(\phi)=P_{m_t}(\phi+t\psi)-\int t \psi dm_t=p(t)-\int t \psi dm_t=p(t)-tp'(t)=p(t)-ta.$$
For all other $\mu\in\mathfs M(\Sigma^+)$ such that $\int\psi d\mu=a$, we have
$$
P_{\mu}(\phi)=P_\mu(\phi+t\psi)-\int t \psi d\mu\leq p(t)-\int t \psi d\mu=p(t)-ta=P_{m_t}(\phi).
$$
So $P_{m_t}(\phi)=\sup\{P_\mu(\phi):\mu\in\mathfs M(\Sigma^+), \int \psi d\mu=a\}=q(a)$,  proving the claim.

\medskip
Let $\delta:=\delta_\theta(\phi):=\frac{\eps}{2M^2}$, and  $I_\psi:=(a_0-\delta\sigma^4, a_0+\delta \sigma^4)$. The claim shows that every $a\in I_\psi$ equals $p'(t)$ for some $|t|\leq \frac{\eps\sigma^2}{2M^2}$  and $q(a)=P_{m_t}(\phi)=p(t)-ta$ is uniformly bounded. It follows that $I_\psi\subset (p'(-\infty),p'(+\infty))$, the domain of $q$, and  $|q|$ is uniformly bounded on $I_\psi$. This proves part (1) of the theorem.

\medskip
When we proved the claim, we mentioned in passing the following fact:
\begin{equation}\label{Legendre-Transform}
q(a)=p(t)-tp'(t)\text{ for the unique $|t|\leq \eps$ such that $p'(t)=a$.}
\end{equation}
\eqref{Legendre-Transform} says that  $q:I_\psi\to\R$  is minus the Legendre transform of the restriction of $p$ to  $(p')^{-1}(I_\psi)$.
Indeed, we have the identity  $q=(p-id\cdot p')\circ(p')^{-1}$.

By Theorem \ref{t.TDF-for-SPR},  $p'$ is $C^\infty$ on $(-\eps,\eps)$, and by Theorem \ref{t.sigma} $p''\neq 0$ (because $p''(t)=0$ $\Rightarrow$ $\sigma_{m_t}(\psi)=0$ $\Rightarrow$ $\psi$ is cohomologous to a constant $\Rightarrow$ $\sigma_m(\psi)=0$ in contradiction to our assumptions).
It follows that  $(p')^{-1}:I_\psi\to (-\eps,\eps)$ is $C^\infty$,  whence $q\equiv (p-id\cdot p')\circ (p')^{-1}$ is $C^\infty$ on $I_\psi$. Since we have already seen that $q$ is uniformly bounded on $I_\psi$, this proves part (2).

By \eqref{Legendre-Transform}, $q(p'(t))=p(t)-tp'(t)$. Repeated differentiation with respect to $t$ gives
\begin{equation}\label{eq.Legendre-Identities}
q'(p'(t))=-t,\ q''(p'(t))=-1/p''(t),\ q'''(p'(t))=p'''(t)/p''(t)^3.
\end{equation}
So
$
q''(a)=q''(p'(t))=-\frac{1}{p''(t)}\in (-2\sigma^{-2}, -\tfrac{1}{2}\sigma^{-2})
$ by \eqref{t-estimate},
 and $|q'''(a)|=\left|\frac{p'''(t)}{p''(t)^3}\right|\leq 8M\sigma^{-6}$.   Part (3) follows  with
 $
 H_\beta(\phi):=8M.
 $
  Part (4) also follows from \eqref{eq.Legendre-Identities}, because
   $p(0)=P_G(\phi)$, $p'(0)=a_0$ and $p''(0)=\sigma^2$.
\end{proof}

\begin{corollary}\label{c.h-double-ineq}
Under the assumptions of Lemma~\ref{lemma-h},
there are $\delta_\theta(\phi),H_\theta(\phi)>0$ such that if  $0<\delta\leq \delta_\theta(\phi)$,
then the following holds  for every $\psi\in\mathcal H_\beta$ such that
$\|\psi\|_\beta\leq 1$ and $\sigma_{m}(\psi)\neq 0$. Let $a_0:=\int\psi dm$, $\sigma:=\sigma_m(\psi)$.
\begin{enumerate}[(1)]
\item  If  $|a-a_0|\leq\frac{\delta}{H_\theta(\phi)}\sigma^4$, then
$$
e^{-\delta}\frac{1}{2\sigma^2}\left(a-a_0\right)^2
\leq \mathfrak{q}_{\phi,\psi}(a_0)-\mathfrak{q}_{\phi,\psi}(a)\leq
e^{\delta}\frac{1}{2\sigma^2}\left(a-a_0\right)^2.
$$
\item If $|a-a_0|>\frac{\delta_\theta(\phi)}{H_\theta(\phi)}\sigma^4$, then
$
\mathfrak{q}_{\phi,\psi}(a_0)-\mathfrak{q}_{\phi,\psi}(a)\geq \frac{\delta_\theta(\phi)\sigma^2}{8H_\theta(\phi)}|a-a_0| .
$
\end{enumerate}
\end{corollary}

\begin{proof}
Let $\sigma:=\sigma_{m}(\psi), a_0:=\int\psi dm$,
$q:={\q}_{\phi,\psi}$, $p:={\mathfrak p}_{\phi,\psi}$.
Let $H:=H_\theta(\phi)$ and  $\delta_\theta(\phi)$ be as in the previous lemma.
Without loss of generality $\delta_\theta(\phi)<\frac{1}{3}$ and $H>1$ (actually, the proof of Lemma~\ref{lemma-h} gives $\delta_\theta(\phi)=\frac{\eps}{2M^2}<10^{-6}$ and $H=8M>800$). Let's write $A=B\pm C$ if $A\in[B-|C|,B+|C|]$.

  Suppose $|a-a_0|\leq \frac{\delta\sigma^4}{H}$ where $0<\delta\leq \delta_\theta(\phi)$. Then $a\in I_\psi$ and we can use the properties listed in Lemma~\ref{lemma-h}.
Taylor's expansion gives
\begin{align*}
&q(a)=q(a_0)+q'(a_0)(a-a_0)+\frac{1}{2}q''(a_0)(a-a_0)^2+\frac{1}{6}q'''(\eta)(a-a_0)^3
\end{align*}
 for some $\eta$ such that $|\eta-a_0|\leq \frac{\delta\sigma^4}{H}$.
 By Lemma~\ref{lemma-h},
\begin{align*}
&q(a)=q(a_0)-\frac{1}{2\sigma^2}(a-a_0)^2\pm\frac{1}{6}\frac{H}{\sigma^6}|a-a_0|^3\\
&=q(a_0)-\frac{1}{2\sigma^2}(a-a_0)^2\biggl(1\pm\frac{1}{3}\frac{H}{\sigma^4}|a-a_0|\biggr)
=q(a_0)-\frac{1}{2\sigma^2}(a-a_0)^2(1\pm \tfrac{\delta}{3})\\
&=q(a_0)-e^{\pm\delta}\frac{1}{2\sigma^2}(a-a_0)^2,
\end{align*}
where the last bound is because $0<\delta<\delta_\beta(\phi)<\tfrac{1}{3}$,
so  $(1-\tfrac{\delta}{3}, 1+\tfrac{\delta}{3})\subset (e^{-\delta},e^{\delta})$.
Rearranging terms, we obtain  the first part of the corollary.

\medskip
The second part of the corollary is more delicate,  because it deals with parts of the domain where we do not know that $q(\cdot)$ is differentiable.

Suppose $a-a_0 > \frac{\delta\sigma^4}{H}$ with $\delta=\delta_\theta(\phi)$, and  let $a_1:=a_0+\frac{\delta\sigma^4}{2H}$,
then $a-a_1\geq \frac{1}{2}(a-a_0)$.
 Since $\delta:=\delta_\theta(\phi)$, $q(\cdot)$ is $C^\infty$ on a neighborhood of $[a_0,a_1]$, and
 $
q''\leq -\frac{1}{2\sigma^2}\text{ on }[a_0,a_1].
 $
So by the mean value theorem for $q'$,
\begin{align*}
&q'(a_1)=q'(a_0)+q''(\xi)(a_1-a_0)\text{ for some }\xi\in [a_0,a_1]\\
&\leq -\frac{1}{2\sigma^2}(a_1-a_0)=-\frac{\delta\sigma^2}{4H}.
\end{align*}
Although we cannot assume that $q$ is differentiable on $[a_0,a]$, we do know that it is concave there. This is sufficient to deduce that
$\frac{q(a)-q(a_1)}{a-a_1}\leq (D^+ q)(a_1)=q'(a_1)\leq -\frac{\delta\sigma^2}{4H}$.
Rearranging terms, and recalling that $(a-a_1)\geq \frac{1}{2}(a-a_0)$, we find that
\begin{align*}
q(a)&\leq q(a_1)-\frac{\delta\sigma^2}{4H}(a-a_1)\leq q(a_0)-\frac{\delta\sigma^2}{8H}(a-a_0),
\end{align*}
where the last inequality uses the inequality $q(a_1)\leq q(a_0)$, a consequence of part (1). Rearranging terms, we obtain part (2) in the case when $a>a_0+\frac{\delta\sigma^4}{H}$.
The case $a<a_0-\frac{\delta\sigma^4}{H}$ is obtained from the symmetry $\psi\leftrightarrow-\psi$.
\end{proof}

\section{The EKP Inequality for measures with large pressure}\label{sec:large_pressure}
Suppose $\Sigma^+$ is a {topologically transitive} countable Markov shift with finite Gurevich entropy.
Let $\phi$ be a $\theta$-weakly H\"older continuous SPR potential such that $\sup\phi<\infty$.
Let $m$ be the unique equilibrium measure of $\phi$. Suppose $e^{-\beta}\leq \theta$, and recall that $(\mathcal H_\beta,\|\cdot\|_\beta)$ denotes the space of $\beta$-H\"older continuous functions of $\Sigma^+$, see \eqref{eq:H-beta}.

\begin{theorem}\label{t.optimal-EKP}
There exist $\eps_\theta^\ast(\phi),C_\theta^\ast(\phi)>0$ such that for every $0\not\equiv\psi\in\mathcal H_\beta$, $0<\eps<\eps_\theta^\ast(\phi)$,
and  $\mu\in\mathfs M(\Sigma^+)$, if
$
P_\mu(\phi)\geq P_G(\phi)-C_\theta^\ast(\phi)\eps^2\frac{\sigma_{m}^6(\psi)}{\|\psi\|_\beta^6}
$
then
$$
\left|\int\psi d\mu-\int\psi dm\right|\leq \sqrt{2}e^{\eps}\sigma_m(\psi)\sqrt{P_G(\phi)-P_\mu(\phi)}.
$$
The bound is sharp in the following sense:
For any $\psi\in\mathcal H_\beta$ such that $\sigma_{m}(\psi)>0$ there exists a sequence of ergodic measures $\nu_n\in\mathfs M(\Sigma^+)$ such that $P_{\nu_n}(\phi)\to P_G(\phi)$, and
\[
\displaystyle\frac{\left|\int\psi d\nu_n-\int\psi dm\right|}{\sqrt{P_G(\phi)-P_{\nu_n}(\phi)}}\xrightarrow[n\to\infty]{}\sqrt{2}\sigma_{m}(\psi).
\]
\end{theorem}
\noindent
\medskip
{\em Remark.\/} Recall from \S\ref{s.setup} that in the special case $\phi\equiv 0$, $m$ is the measure of maximal entropy,  $P_G(\phi)$ is the entropy of $m$, and the condition that $\phi$ is SPR is the same as the condition that $\Sigma^+$ is SPR.

\begin{proof}
Suppose first that $\sigma:\Sigma^+\to\Sigma^+$ is topologically {\em mixing}.

If $\sigma_{m}(\psi)=0$, then $P_\mu(\phi)\geq P_G(\phi)-C_\theta^\ast(\phi)\eps^2\frac{\sigma_{m}^6(\psi)}{\|\psi\|_\beta^6}$ implies that $P_\mu(\phi)=P_G(\phi)$, whence by the uniqueness of the equilibrium measure $\mu=m$ and the inequality is trivial. So it is enough to consider the case $\sigma_m(\psi)\neq 0$.
It is easy to verify that $\sigma_{m}(t\psi)=|t|\sigma_{m}(\psi)$.
This allows us to work with normalized functions $\psi/\|\psi\|_\beta$.
Henceforth we assume that $\|\psi\|_\beta=1$ and $\sigma:=\sigma_{m}(\psi)\neq 0$.
Let $a_0:=\int \psi dm$. Let $H_\theta(\phi)$ and $\delta_\theta(\phi)$ denote the constants from Corollary~\ref{c.h-double-ineq}.
Let
$
H:=H_\theta(\phi)$, $C^\ast:=\frac{1}{9H^2}.
$
Fix some $0<\delta<\delta_\theta(\phi)$, and
suppose $\mu\in\mathfs M(\Sigma^+)$ satisfies  $P_\mu(\phi)\geq P_G(\phi)-C^\ast\delta^2\sigma^6$.
Let $a:=\int\psi d\mu$.
By the definition of the restricted pressure,
$$\mathfrak{q}_{\phi,\psi}(a)\geq P_\mu(\phi).$$
We claim that $|a-a_0|\leq \frac{\delta_\theta(\phi)\sigma^4}{H}$. Otherwise, by the assumption on $\mu$,
\begin{align*}
&C^\ast\delta^2_\theta(\phi)\sigma^6 \geq P_G(\phi)-P_\mu(\phi)\geq P_G(\phi)-\mathfrak{q}_{\phi,\psi}(a),\text{ because }\mathfrak{q}_{\phi,\psi}(a)\geq P_\mu(\phi)\\
&=\mathfrak{q}_{\phi,\psi}(a_0)-\mathfrak{q}_{\phi,\psi}(a),\text{ because $\mathfrak{q}_{\phi,\psi}(a_0)=P_G(\phi)$ by Lemma \ref{lemma-h}} \\
&\geq \frac{\delta_\theta(\phi)\sigma^2}{8H}|a-a_0|\text{ by the 2nd part of  Corollary~\ref{c.h-double-ineq}}\\
&\geq \frac{\delta_\theta(\phi)\sigma^2}{8H}\frac{\delta_\theta(\phi)\sigma^4}{H}=\frac{\delta^2_\theta(\phi)\sigma^6}{8 H^2}\text{, by the assumption $|a-a_0|\geq \frac{\delta_\theta(\phi)\sigma^4}{H}$.}
\end{align*}
But this  contradicts the definition of $C^\ast$.

So  $|a-a_0|\leq \frac{\delta_\theta(\phi)\sigma^4}{H}$, and the first part of Corollary~\ref{c.h-double-ineq} gives us
$$
(a-a_0)^2\leq 2\sigma^2 e^\delta (\mathfrak{q}_{\phi,\psi}(a_0)-\mathfrak{q}_{\phi,\psi}(a)).
$$
Taking the square root, and recalling that $a=\int\psi d\mu$, $a_0=\int \psi dm$,
$\q_{\phi,\psi}(a_0)=P_G(\phi)$, and $\q_{\phi,\psi}(a)\geq P_\mu(\phi)$, we obtain
\begin{align*}
&\left|\int\psi d\mu-\int\psi dm\right|\leq e^{\delta}\sqrt{2}\sigma\sqrt{P_G(\phi)-\mathfrak{q}_{\phi,\psi}(a)}\leq
e^{\delta}\sqrt{2}\sigma\sqrt{P_G(\phi)-P_\mu(\phi)}.
\end{align*}
This proves
 the first part of the theorem with $\eps^\ast_\theta(\phi):=\delta_\theta(\phi)$.

To see the second part, take $a_n\to a_0$. When we proved Lemma~\ref{lemma-h}, we saw that if $a_n$ is sufficiently close to $a_0$, then $\exists! t_n$ such that $\p_{\phi,\psi}'(t_n)=a_n$, the equilibrium measure $\nu_n:=m_{t_n}$ of $\phi+t_n\psi$ exists, and $\mathfrak{q}_{\phi,\psi}(a_n)=P_{\nu_n}(\phi)$.

 Repeating  the previous argument with $\nu_n$ replacing $\mu$, but now with the full force of Corollary~\ref{c.h-double-ineq}(1), we find that
$
\sqrt{2}\sigma e^{-\delta_n}\leq \frac{\left|\int\psi d\nu_n-\int\psi dm\right|}{\sqrt{P_G(\phi)-P_{\nu_n}(\phi)}}\leq \sqrt{2}\sigma e^{\delta_n},
$
where $\delta_n:=\frac{H}{\sigma^4}|a_n-a_0|\to 0$.

\medskip
This completes the proof in the topologically mixing case. We will now outline the proof in the non-mixing topologically transitive case.
Suppose $\sigma:\Sigma^+\to\Sigma^+$ is  topologically transitive, with period $p$, and let
$$
\Sigma^+=\Sigma_0^+\uplus\cdots\uplus\Sigma_{p-1}^+
$$
be the spectral decomposition from  \S\ref{s.spectral-decomposition}. The assumption that $\phi$ is SPR on $\Sigma^+$ means, by definition, that $\phi_p:=\sum_{k=0}^{p-1}\phi\circ\sigma^k$ is SPR with respect to the topologically mixing $\sigma^p:\Sigma_i^+\to\Sigma_i^+$ for some (and then for all) $i=0,\ldots,p-1$.

$\Sigma_i^+=\sigma^{i-j}(\Sigma_j^+)$, therefore for every  $\sigma$-invariant measure $\mu$, $\mu(\Sigma_i^+)$ are all equal (to $1/p$), and
$
\mu=\frac{1}{p}\sum_{i=0}^{p-1}\mu_i,\text{ where }\mu_i:=\mu(\cdot | \Sigma_i^+)
$ (the conditional measure on $\Sigma_i^+$).
In addition:
\begin{enumerate}[(1)]
\item $h_{\mu_i}(\sigma^p)=p h_\mu(\sigma)$: Firstly, $\sigma^{j+p-i}:(\Sigma_i^+,\mu_i)\to (\Sigma_j^+,\mu_j)$ is a factor map for all $i,j$ so $h_{\mu_i}(\sigma^p)$ are all equal. Secondly, by the affinity of the entropy map, $\frac{1}{p}\sum_{i=0}^{p-1} h_{\mu_i}(\sigma^p)=h_\mu(\sigma^p)=ph_\mu(\sigma)$.
\item  $\int\phi_p d\mu_i=p\int\phi d\mu$. {(More generally, for any integer multiple of $p$.)}
\item By (1) and (2), $P_{\mu_i}(\phi_p)=pP_\mu(\phi)$.
\item By definition, $P_G(\phi_p|_{\Sigma_i^+},\sigma^p)=pP_G(\phi|_{\Sigma^+},\sigma)$.
\item If $m$ is an equilibrium measure of $\phi$, then $m_i$ is an equilibrium measure of $\phi_p$.
\item By definition, $\sigma_{m_i}^2(\phi_p) = \lim_{n\to\infty} \frac1n \Var_{m_i}(\sum_{j=0}^{n-1}\psi_p\circ\sigma^{jp})$.
\item $\sigma_{m_0}(\psi_p)=\sqrt{p}\sigma_m(\psi)$: By (2), $\E_{m_i}(\psi_{np})=np\E_m(\psi)=\E_m(\psi_{np})$, therefore
\begin{align*}
&\sigma^2_m(\psi)=\lim_{n\to\infty}\frac{1}{np}\Var_m(\psi_{np})
=\lim_{n\to\infty}\frac{1}{np^2}
\sum_{i=0}^{p-1}\Var_{m_i}(\psi_{np})\\
&=\lim_{n\to\infty}\frac{1}{np^2}
\sum_{i=0}^{p-1}\Var_{m_0}\bigl(\sum_{j=0}^{n-1}(\psi_{p})\circ\sigma^{jp}\circ\sigma^i\bigr)=
\frac{1}{p^2}\sum_{i=0}^{p-1}\sigma_{m_0}^2(\psi_p\circ\sigma^i)=\frac{1}{p}\sigma_{m_0}^2(\psi_p),
\end{align*}
where the last equality is because $\psi_p\circ \sigma^i$ is $\sigma^p$-cohomologous to $\psi_p$.
\end{enumerate}
It is now an easy exercise to deduce the theorem for $\sigma:\Sigma^+\to\Sigma^+,\psi,m,\mu$ from the theorem for the topologically mixing $\sigma^p:\Sigma_0^+\to
\Sigma_0^+$, with  $\psi_p, m_0,\mu_0$.
\end{proof}

\medskip
The previous result gives the ``optimal" form of the EKP inequality for measures with high entropy. Note that the bound   $\const\cdot\|\psi\|_\beta\sqrt{h-h_{\mu}(\sigma)}$  in the original EKP inequality is replaced by
\[
e^{\eps}\sqrt{2}\sigma_{\mu_0}(\psi)\sqrt{h-h_\mu(\sigma)},\]
where $\sigma_{\mu_0}^2(\psi)$ is the asymptotic variance of $\psi$ with respect to the measure of maximal entropy $\mu_0$.
This is better than \eqref{EKP-ineq}, because $\sigma_{\mu_0}(\psi)\leq M\|\psi\|_\beta$ (Theorem~\ref{t.sigma}), and because of the following fact.

\begin{lemma}\label{lem:small_sigma}
 Suppose $\Sigma^+$ is topologically transitive, with positive Gurevich entropy. For every $\beta>0$, there is a sequence of $\psi^{(n)}\in\mathcal H_\beta$ such that $\sigma_{\mu_0}(\psi^{(n)})\neq 0$, $\int\psi^{(n)} d\mu_0=0$, and
$\displaystyle
\frac{\sigma_{\mu_0}(\psi^{(n)})}{\|\psi^{(n)}\|_\beta}\xrightarrow[n\to\infty]{}0.
$
\end{lemma}
\begin{proof}
There are two  periodic points $\un{x},\un{y}$ with the same (perhaps non-minimal) period $p$, and with disjoint orbits: $\sigma^m(\un{x})\neq \sigma^n(\un{y})$ for all $m,n$.
Otherwise the Gurevich entropy equals zero.

Construct $\psi\in\mathcal H_\beta$ such that
$
\displaystyle\sum_{j=0}^{p-1}\psi(\sigma^j(\un{x}))\neq \sum_{j=0}^{p-1}\psi(\sigma^j(\un{y}))
$ and $\int\psi d\mu_0=0$.
Then $\psi$ cannot be cohomologous to a constant
(otherwise it would give all periodic points of fixed period the same weight).
By Theorem~\ref{t.sigma}, $\sigma_{\mu_0}(\psi)\neq 0$.
In addition,
since $\psi$ is clearly non-constant,
$\psi\neq\psi\circ\sigma$,
whence $\|\psi-\psi\circ\sigma\|_\beta\neq 0$.

Now take $\psi^{(n)}:=\frac{1}{n}\psi+(\psi-\psi\circ\sigma)$.
On the one hand $\sigma_{\mu_0}(\psi^{(n)})=\sigma_{\mu_0}(\psi/n)=\sigma_{\mu_0}(\psi)/n$, a sequence of positive numbers which converges to zero.
On the other hand $\|\psi^{(n)}\|_\beta\geq \|\psi-\psi\circ\sigma\|_\beta-\|\psi/n\|_\beta\to \|\psi-\psi\circ\sigma\|_\beta\neq 0$.
\end{proof}

We can now explain why we needed the constants $\eps_\theta(\phi), M_\theta(\phi)$  in  Theorem \ref{t.TDF-for-SPR} to be independent of  $\psi$.
We do this in the  case of main interest  $\phi\equiv 0$, when $\mu_0$ is the measure of maximal entropy. In this case Theorem \ref{t.optimal-EKP} gives a nearly optimal EKP inequality in the regime
$$
h_\mu(\sigma)\geq h-\eps^2C^\ast_\theta(0)(\sigma_{\mu_0}(\psi)/\|\psi\|_\beta)^6.
$$
In the absence of the uniformity in $\psi$ in Theorem \ref{t.TDF-for-SPR} the best we could have hoped for was to prove this bound in the regime
$
h_\mu(\sigma)\geq h-\eps^2C^\ast_\theta(\psi),
$
but without further information on the structure of $C^\ast_\theta(\psi)$.

\section{The EKP inequality for arbitrary measures}
Our next result (which reduces in the case of subshifts of finite type and $\phi\equiv 0$ to a result of S.\ Kadyrov \cite{Kadyrov-Effective-Uniqueness}) is an inequality for all $\sigma$-invariant measures, also those with low entropy or pressure.

Suppose $\Sigma^+$ is a topologically mixing countable Markov shift with finite and positive Gurevich entropy. Let $\phi$ be a $\theta$-weakly H\"older continuous SPR potential such that $\sup\phi<\infty$. Let $m$ denote the unique equilibrium measure of $\phi$. Fix $\beta$ such that $e^{-\beta}\leq \theta$, and let $(\mathcal H_\beta,\|\cdot\|_\beta)$ denote the space of $\beta$-H\"older continuous functions on $\Sigma^+$, see \eqref{eq:H-beta}.

\begin{lemma}
There exist constants $K_\beta$, $Q(\phi)>0$ such that for every $\sigma$-invariant probability $\mu$ there exists some function $A\in\mathcal H_\beta$ such that $\int A dm=\int A d\mu=0$, $\|A\|_\beta\leq K_\beta$, and $\sigma_{m}(A)>Q(\phi)$.
\end{lemma}
\begin{proof}

As in Lemma~\ref{lem:small_sigma}, we can find $p\geq 1$ and four periodic points $\un{x},\un{y},\un{z},\un{w}\in \Sigma^+$ with period $p$, such that
$$
x_0=y_0=z_0=w_0
$$
and so that the orbits of $\un{x}$, $\un{y}$, $\un{z}$, $\un{w}$ are disjoint.
Let $\un{x}^{\text{p}}:=(x_0,\ldots,x_{p-1},x_0)$,
$\un{y}^{\text{p}}:=(y_0,\ldots,y_{p-1},y_0)$,
$\un{z}^{\text{p}}:=(z_0,\ldots,z_{p-1},z_0)$,
$\un{w}^{\text{p}}:=(w_0,\ldots,w_{p-1},w_0)$.
Define
$$a(\cdot):=1_{[\un{x}^{\text{p}}]}(\cdot)-m([\un{x}^{\text{p}}]),\  b(\cdot):=1_{[\un{z}^{\text{p}}]}(\cdot)-m([\un{z}^{\text{p}}]).$$

Recall the notation $\vf_n=\sum_{k=0}^{n-1}\vf\circ\sigma^k$.
Since $\{\sigma^k\un{x}\}_{k\in\N}$, $\{\sigma^k\un{y}\}_{k\in\N}$, $\{\sigma^k\un{z}\}_{k\in\N}$, $\{\sigma^k\un{w}\}_{k\in\N}$ are disjoint,
the orbit of $\un{y}$ does not enter $[\un{x}^{\text{p}}]$,
and the orbit of $\un{w}$ does not enter $[\un{z}^{\text{p}}]$. In particular, $a_p(\un{x})\neq a_p(\un{y})$ and $b_p(\un{z})\neq b_p(\un{w})$.

This implies that  $a,b$ are not cohomologous to constants
(otherwise  $a_p(\un{x})=a_p(\un{y})=\const p$ for any pair of $p$-periodic orbits).
So $\sigma_{m}^2(a),\sigma_{m}^2(b)\neq 0$.

By construction,  $\int a dm=\int b dm=0$.
 If $\int a d\mu=0$ take $A:=a$.
 If $\int b d\mu=0$, take $A:=b$.
Notice that $A$ is independent of $\mu$, therefore
$\|A\|_\beta$, $\sigma_{m}^2(A)$ are independent of $\mu$,
and the lemma follows.

In the remaining case,
$\int a d\mu\neq 0$ and $\int b d\mu\neq 0$,
and we define
$$
A:=\frac{a\int b d\mu-b\int a d\mu}{\sqrt{(\int a d\mu)^2+(\int b d\mu)^2}}.
$$
Clearly $\int A d\mu=\int A dm=0$, and
\begin{align*}
\|A\|_\beta&\leq \frac{|\int a d\mu|+|\int b d\mu|}{\sqrt{(\int a d\mu)^2+(\int b d\mu)^2}}\max\{\|a\|_\beta,\|b\|_\beta\}\leq 2\max\{\|a\|_\beta,\|b\|_\beta\}=:K_\beta.
\end{align*}
$K_\beta$ is independent of $\mu$,
and only depends on $\beta$ and $\Sigma^+$.
To complete the proof,
it remains to  bound $\sigma_{m}^2(A)$ from below by a constant
which is independent of $\mu$.

We need for this purpose the function $Q(s,t):=\sigma_{m}^2(sa+tb)$.
By Theorem~\ref{t.TDF-for-SPR},
\begin{align*}
&\sigma_{m}^2(sa+tb)=\frac{d^2}{d\tau^2}\biggr|_{\tau=0}
\p_{\phi,s a+t b}(\tau)\equiv \frac{d^2}{d\tau^2}\biggr|_{\tau=0}
P_G(\phi+\tau s a+\tau t b)\\
&\overset{!}{=} s^2\frac{\partial^2 P}{\partial u^2}(0,0)+2st
\frac{\partial^2 P}{\partial u\partial v}(0,0)+t^2 \frac{\partial^2 P}{\partial v^2}(0,0),\text{ where }P(u,v):=P_G(\phi+ua+vb).
\end{align*}
To justify $\overset{!}{=}$, we use Theorem~\ref{t.TDF-for-SPR}~(6).
We see that $Q(s,t)$ is a quadratic form.

By the definition of the asymptotic variance, $Q(s,t)\geq 0$ for all $(s,t)$.
We claim that $Q$  is positive definite.

Suppose $Q(s,t)=0$, then
$\sigma_{m}^2(sa+tb)=0$, whence $sa+tb$ is cohomologous to a constant. In this case, by Livshits theorem, $sa+tb$  gives the same weight to all periodic orbits with the same period. In particular
\begin{align*}
sa_p(\un{x})+tb_p(\un{x})&=sa_p(\un{y})+tb_p(\un{y})\\
sa_p(\un{z})+tb_p(\un{z})&=sa_p(\un{w})+tb_p(\un{w}).
\end{align*}
Recall that the orbits of $\un{x}$, $\un{y}$, $\un{z}$, $\un{w}$ are disjoint,
so the orbits of $\un{y}$, $\un{z}$, $\un{w}$ do not enter
$[\un{x}^{\text{p}}]$
whence $a_p(\un{y})=a_p(\un{z})=a_p(\un{w})=-pm([\un{x}^{\text{p}}])$,
and the orbits of $\un{x}$, $\un{y}$, $\un{w}$
do not enter $[\un{z}^{\text{p}}]$, so
$b_p(\un{x})=b_p(\un{y})=b_p(\un{w})=-pm([\un{z}^{\text{p}}])$.
On the other hand $a_p(\un{x})=n_{\un{x}}-pm([\un{x}^{\text{p}}]), b_p(\un{z})=n_{\un{z}}-pm([\un{z}^{\text{p}}])$
with $n_{\un{x}},n_{\un{z}}$ positive integers.
Substituting this above, we obtain $n_{\un{x}} s=0, n_{\un{z}} t=0$,
whence $s=t=0$.
So $Q(s,t)=0\Rightarrow (s,t)=(0,0)$,
and $Q$ is positive definite.

Since $Q(s,t)$ is a positive definite quadratic form, there exists $Q_0>0$ such that $Q(s,t)\geq Q_0^2(s^2+t^2)$ for all $(s,t)\in\R^2$. In particular,
\begin{align*}
\sigma_{m}^2(A)&=Q\left(\frac{\int b d\mu}{\sqrt{(\int a d\mu)^2+(\int b d\mu)^2}},
\frac{-\int a d\mu}{\sqrt{(\int a d\mu)^2+(\int b d\mu)^2}}
\right)\geq Q_0^2.
\end{align*}
Notice that $Q_0$ depends only on $a,b$ and $\phi$,
and is therefore independent of $\mu$. We let $Q(\phi):=Q_0$.
\end{proof}

\begin{lemma}
Given $\beta>0$   $\exists K_\beta'(\phi)>0$ as follows: For every $\psi\in\mathcal H_\beta$ and $\mu\in\mathfs M(\Sigma^+)$,  $\exists\vf\in\mathcal H_\beta$ such that $\int\vf dm=\int\psi dm$,  $\int\vf d\mu=\int\psi d\mu$, $\|\vf\|_\beta\leq K_\beta'(\phi)\|\psi\|_\beta$ and
\begin{equation}\label{varphi}
\frac{\|\vf\|_\beta}{\sigma_{m}(\vf)}\leq K_\beta'(\phi).
\end{equation}
\end{lemma}
\begin{proof}
Let $Q:=Q(\phi)$, $K_\beta$ and $A(\cdot)$ be as in the previous lemma.

If  $\sigma_{m}(\psi)\geq \frac{1}{3}Q\|\psi\|_\beta$,
we take $\vf:=\psi$, and note that
$$
\frac{\|\vf\|_\beta}{\sigma_{m}(\vf)}\leq \frac{3}{Q}.
$$

If $\sigma_{m}(\psi)<\frac{1}{3}Q\|\psi\|_\beta$,
we take $\vf:=\psi+\|\psi\|_\beta A$.
Then
$
\int\vf dm=\int\psi dm$, $\int\vf d\mu=\int\psi d\mu$,
and $\|\vf\|_\beta\leq (1+\|A\|_\beta)\|\psi\|_\beta\leq (1+K_\beta)\|\psi\|_\beta
$.
In addition,
\begin{align*}
\sigma_{m}^2(\vf)&=\lim_{n\to\infty}\frac{1}{n}\Var(\vf_n),\text{ where }\vf_n:=\sum_{k=0}^{n-1}\vf\circ\sigma^k\text{, }\Var(B):=\int(B-{\textstyle\int} Bdm)^2 dm\\
&=\lim_{n\to\infty}\frac{1}{n}\Var(\psi_n)
+\lim_{n\to\infty}\frac{1}{n}\Var(\|\psi\|_\beta A_n)+2\lim_{n\to\infty}\frac{1}{n}\mathrm{Cov}(\psi_n, \|\psi\|_\beta A_n)\\
&\hspace{2cm}\text{ where }\mathrm{Cov}(B,C):=\int(B-{\textstyle\int} B dm)(C-{\textstyle\int} C dm)dm\\
&\geq 0+\|\psi\|_\beta^2\sigma_{m}^2(A)-2\|\psi\|_\beta\lim_{n\to\infty}\frac{1}{n}
\sqrt{\Var(\psi_n)\Var(A_n)}\text{\ \ \  (Cauchy-Schwarz)}\\
&= \|\psi\|_\beta^2\sigma_{m}^2(A)-2
\|\psi\|_\beta\sigma_{m}(\psi)\sigma_{m}(A) \ \ \\
&\geq \frac{1}{3}
Q^2\|\psi\|_\beta^2 \ \ \ (\because \sigma_{m}(A)\geq Q,\  \sigma_{m}(\psi)<\tfrac{1}{3}Q\|\psi\|_\beta),
\end{align*}
so $\sigma_{m}(\vf)\geq \frac{1}{\sqrt{3}}Q\|\psi\|_\beta$.

We saw above that
$\|\vf\|_\beta\leq (1+K_\beta)\|\psi\|_\beta$. It follows that
$
\frac{\|\vf\|_\beta}{\sigma_{m}(\vf)}\leq \sqrt{3}\frac{1+K_\beta}{Q}
$.
The lemma follows with $K_\beta'(\phi):=\max\{\frac{3}{Q(\phi)}, \sqrt{3}\frac{1+K_\beta}{Q(\phi)}, K_\beta+1\}$.
\end{proof}

\begin{theorem}
\label{t.suboptimal-EKP}
Suppose $\Sigma^+$ is a topologically transitive countable Markov shift with finite Gurevich entropy.
Let $\phi$ be a $\theta$-weakly H\"older continuous SPR potential such that $\sup\phi<\infty$, and let $m$ be the unique equilibrium measure of $\phi$. If $e^{-\beta}\leq \theta$, then $\exists C_{\theta,\beta}(\phi)>0$ such that for every $\psi\in\mathcal H_\beta$, and every  $\mu\in\mathfs M(\Sigma^+)$,
$$
\left|\int\psi d\mu-\int\psi dm\right|\leq C_{\theta,\beta}(\phi)\|\psi\|_\beta\sqrt{P_G(\phi)-P_\mu(\phi)}.
$$
\end{theorem}
\medskip
\noindent
{\em Remark.\/} In the special case $\phi\equiv 0$, $m$ is the measure of maximal entropy, $P_G(\phi)$ is the entropy of $m$, and the inequality becomes \eqref{EKP-ineq}.

\begin{proof}
{As in the proof of Theorem \ref{t.optimal-EKP}, it is sufficient to prove the theorem in the topologically mixing case.}

Fix a $\sigma$-invariant  measure $\mu$ and some $\psi\in\mathcal H_\beta$.
If $\int\psi d\mu=\int\psi dm=0$ then there is nothing to prove so suppose one of these integrals is non-zero.
Let $K:=K_\beta'(\phi)$ be a constant independent of $\mu,\psi$ as in the previous lemma, and  choose $\vf\in\mathcal H_\beta$ such that $\|\vf\|_\beta\leq K\|\psi\|_\beta$, $\int\psi d\mu=\int\vf d\mu$, $\int\psi dm=\int\vf dm$, and
$$
\frac{\|\vf\|_\beta}{\sigma_{m}(\vf)}\leq K.
$$
Notice that $\|\vf\|_\beta\neq 0$, because at least one of the integrals $\int\vf d\mu,\int\vf dm$ is non-zero.
So we can normalize  $\displaystyle \overline{\vf}:=\frac{\vf}{\|\vf\|_\beta}.$

Let $\delta:=\delta_\theta(\phi)$, $H:=H_\theta(\phi)$ be the constants in Corollary~\ref{c.h-double-ineq},
and let $M:=M_\theta(\phi)$ be the constant from Theorem~\ref{t.sigma}. Without loss of generality, $e^{\delta}<\sqrt{2}$.
 Set $a:=\int\overline{\vf} d\mu$, $a_0=\int\overline{\vf} dm$,
 $\sigma:=\sigma_{m}(\overline{\vf})$.
If $|a-a_0|\leq\frac{\delta\sigma^4}{H}$, then
$$
\left|\int\overline{\vf} d\mu-\int\overline{\vf} dm\right|\leq e^\delta\sqrt{2}\sigma\sqrt{\q_{\phi,\ov{\vf}}(a_0)-\mathfrak{q}_{\phi,\ov{\vf}}(a)}\overset{!}{\leq}2M\sqrt{P_G(\phi)-P_\mu(\phi)},
$$
where $\overset{!}{\leq}$ is because $\q_{\phi,\ov{\vf}}(a_0)=P_G(\phi)$, $\q_{\phi,\ov{\vf}}(a)\geq P_\mu(\phi)$, and
$\sigma_{m}(\overline{\vf})\leq M\|\overline{\vf}\|_\beta=M$.

Similarly, if $|a-a_0|> \frac{\delta\sigma^4}{H}$,
then by the 2nd part of Corollary~\ref{c.h-double-ineq}
\begin{align*}
&\frac{1}{2}\left|\int\overline{\vf} d\mu-\int\overline{\vf} dm\right|\leq \frac{1}{2}\frac{8 H}{\delta\sigma^2}(P_G(\phi)-P_\mu(\phi))= \frac{4H}{\delta \sigma_{m}^2(\overline{\vf})}(P_G(\phi)-P_\mu(\phi))\\
&=\frac{4H\|\vf\|_\beta^2}{\delta \sigma_{m}^2(\vf)}(P_G(\phi)-P_\mu(\phi))
\leq \frac{4HK^2}{\delta}(P_G(\phi)-P_\mu(\phi)).
\end{align*}
Since $\|\overline{\vf}\|_\beta=1$,
$\frac{1}{2}|\int\overline{\vf} d\mu-\int \overline{\vf} dm|\leq 1$,
and therefore
\begin{align*}
&\frac{1}{2}\left|\int\overline{\vf} d\mu-\int \overline{\vf} dm\right|\leq \sqrt{\frac{1}{2}\left|\int\overline{\vf} d\mu-\int \overline{\vf}dm\right|}\leq \sqrt{\frac{4HK^2}{\delta}(P_G(\phi)-P_\mu(\phi))},
\end{align*}
whence $\left|\int\overline{\vf} d\mu-\int \overline{\vf} dm\right|\leq (4K\sqrt{H/\delta})\sqrt{P_G(\phi)-P_\mu(\phi)}$.

Let $C_{\theta,\beta}':=C'_{\theta,\beta}(\phi):=\max\{2M,4K\sqrt{H/\delta}\}$.
This depends only on $\beta,\theta$ and $\phi$, and
the following inequality holds  no matter the value of $|a-a_0|$:
\begin{equation}\label{eq.phi}
\left|\int\overline{\vf}d\mu-\int\overline{\vf} dm\right|\leq C_{\theta,\beta}'\sqrt{P_G(\phi)-P_\mu(\phi)}.
\end{equation}
By the choice of $\overline{\vf}$, the left-hand side of \eqref{eq.phi} equals  $\left|\int\psi d\mu-\int\psi dm\right|/\|\vf\|_\beta$, so
$
\left|\int\psi d\mu-\int\psi dm\right|\leq C_{\theta,\beta}'\|\vf\|_\beta\sqrt{P_G(\phi)-P_\mu(\phi)}.
$
Since $\|\vf\|_\beta\leq K\|\psi\|_\beta$, the theorem follows with $C_{\theta,\beta}(\phi):=C_{\theta,\beta}' K$.
\end{proof}

\section{SPR is a necessary condition for the EKP inequality}
Let $\sigma:\Sigma^+\to\Sigma^+$ be a topologically {transitive} TMS associated to a countable directed graph $\mathfs G$, and suppose $\phi:\Sigma^+\to\R$ is a function with summable variations and finite Gurevich pressure.

In this section  we prove that if $\mu_\phi$ satisfies the EKP inequality,
 then $\phi$ must be strongly positively recurrent. Ruette's work implies this for $\phi\equiv 0$ \cite{Ruette}, and we will extend her argument to general potentials. As in Ruette's work, the engine of the proof is the following result, which is of independent interest.
\begin{theorem}\label{prop:ruette_generalized}
Suppose $\Sigma^+$ is topologically mixing. If $\phi$ {is not SPR}, then  $\mathfs G$ contains a  subgraph $\mathfs G'$ such that $\Sigma^+(\mathfs G')\subsetneq\Sigma^+(\mathfs G)$, $\sigma:\Sigma^+(\mathfs G')\to \Sigma^+(\mathfs G')$ is topologically mixing,  $\phi|_{\Sigma^+(\mathfs G')}$ is not SPR, and $P_G(\phi|_{\Sigma^+(\mathfs G')})=P_G(\phi|_{\Sigma^+})$.
\end{theorem}
\begin{proof}
The case $\phi\equiv 0$ is done in \cite{Ruette}.

If $\phi$ {is not SPR,} then $\mathfs G$ must have infinitely many vertices, otherwise all bounded H\"older continuous potentials are SPR, see \cite{Sarig-CMP-2001}.

There is no loss of generality in assuming that $\mathfs G$ has at most one edge $a\to b$ for every ordered pair $(a,b)\in S\times S$, and that every vertex $a$ has an outgoing edge $a\to b$ in $E$. Otherwise, pass to the graph $\mathfs G^\ast$ with set of vertices $S^\ast:=\{a\in S:[a]\neq \emptyset\}$ and set of  edges $\{(a,b)\in E:[a,b]\neq \emptyset\}$, then $\Sigma^+(\mathfs G^\ast)=\Sigma^+(\mathfs G)$ and $\mathfs G^\ast$ has the required properties.

 It follows that if $\mathfs G'$ is a proper subgraph of $\mathfs G$, then $\Sigma^+(\mathfs G')\subsetneq\Sigma^+(\mathfs G)$.

By our standing assumptions,  $\Sigma^+(\mathfs G)$ is topologically mixing. Therefore
there exist two closed loops, starting and ending at the same vertex, with co-prime lengths.
	The union of these loops defines a finite  subgraph $\mathfs G_{core}$ of $\mathfs G$, and every strongly connected graph which contains $\mathfs G_{core}$  defines a topologically mixing TMS.

\medskip
We will obtain the graph $\mathfs G'$ from $\mathfs G$, by removing some  edge $a\to b$ outside  $\mathfs G_{\text{core}}$. We  use the following   {\em edge removal procedure} from \cite{Ruette}.

	Let $S_{\text{core}}$ denote the set of  vertices of $\mathfs G_{\text{core}}$.
	Let $\mathcal{E}$ denote the collection of finite paths $\gamma$ with the following properties:
\begin{enumerate}[(1)]
\item the first vertex and the last vertex of $\gamma$ belong to $S_{\text{core}}$;
\item all other  vertices are outside $S_{\text{core}}$;
\item  for every vertex $x_i\in S\setminus S_{\text{core}}$ of $\gamma=(x_0,\dots,x_{n-1})$ either $(x_0,\dots,x_i)$ or $(x_i,\dots,x_{n-1})$ is a $\mathfs G$--geodesic.
\end{enumerate}
Since $\sigma$ is topologically transitive,  $\mathfs G$ is strongly connected. Every vertex in $S\setminus S_{\text{core}}$ can be connected by  forward and backward $\mathfs G$--geodesics to $S_{\text{core}}$. Their union is a path in $\mathcal E$. So every vertex in $S_{\text{core}}^c$ belongs to some path in $\mathcal E$, and $\mathcal E$ is infinite.

All paths in $\mathcal E$ begin at $S_{\text{core}}$. Since $S_{\text{core}}$ is finite and $\mathcal E$ is infinite,	there exist two different paths $\gamma_0,\gamma_1 \in \mathcal{E}$ which begin at the same vertex in $S_{\text{core}}$ and end at the same vertex in $S_{\text{core}}$. Let $\gamma$ denote the maximal common prefix of $\gamma_0,\gamma_1$, and let $a$ denote the last vertex in $\gamma$.

By construction, $a$ has at least two different outgoing edges $a\to b_i$ $(i=0,1)$ such that $b_i\not\in S_{\text{core}}$.
	Divide the outgoing edges from $a$ to $S_{core}^c$ into two arbitrary subsets,
	$E_0$ and $E_1$, such that $(a,b_i)\in E_i$.

	Let $\mathfs G_i$ be the graph obtained by removing the edges $E_{1-i}$ from $\mathfs G$ and restricting to the strongly connected component containing $a$.
It is not difficult to see that $\mathfs G_i$ contain $\mathfs G_{\text{core}}$, so  $\sigma:\Sigma^+(\mathfs G_i)\to\Sigma^+(\mathfs G_i)$ is topologically mixing for $i=1,2$. By the assumption on $\mathfs G$, $\Sigma^+(\mathfs G_i)\subsetneq\Sigma^+(\mathfs G)$.

If $x_0=a,x_1,\dots,x_{n-1},x_n=a$ is a first return loop to $a$, then $x_i\neq a$ for $1\leq i\leq n-1$, therefore edges from $E_i$ can only appear as the first edge $(x_0,x_1)$.

We let $E_{\text{core}}=E_{\text{core}}(a)$ be the edges $(a,b) \in \mathfs G$ such that $a,b\in S_{core}$ (in particular,  $E_{\text{core}}=\emptyset$ whenever $a\not\in S_{core}$).
$Z^*_n(\phi, a)$ splits into the sums,
	\begin{equation}\label{e.Z-splitting}
		Z^*_n(\phi, a)=Z_n^{*}(E_0) + Z_n^*(E_1) + Z_n^*( E_{\text{core}})
	\end{equation}
	where
	\[
	Z_n^{*}(E_i):=\sum_{\substack{\sigma^n(\un{x})=\un{x} \\ (x_0,x_1)\in E_i}} e^{\phi_n(\un{x})} 1_{[\tau_a=n]}(\un{x})\ \ \ (\text{$\tau_a$ is defined in Section~\ref{s.eq}})
	\]
	and
	\[
	Z_n^{*}(E_{\text{core}}):=\sum_{\substack{\sigma^n(\un{x})=\un{x} \\ (x_0,x_1)\in E_\text{core}}} e^{\phi_n(\un{x})} 1_{[\tau_a=n]}(\un{x}).
	\]

We claim that
\begin{equation}\label{eq.id}
Z_n^{*}(\phi|_{\Sigma^+(\mathfs G_i)},a)=Z_n^{*}(E_i)+Z_n^{*}(E_\text{core}).\end{equation}
The non-trivial inequality is ($\geq$). The first return loops at $a$ which begin with an edge in $E_i$ or $E_{\text{core}}$ do not contain other edges in $E_{1-i}$, because $a$ cannot appear in the middle of a first return loop at $a$. Therefore all such loops must be contained in the irreducible component of $a$ in $\mathfs G\setminus E_{1-i}$, whence in $\mathfs G_i$. So the loops participating in the sums in the right-hand-side must also appear in the sum on the left-hand-side. As the sums on the right are over disjoint sets,  $(\geq)$ follows.
%
%
%

Recall that $P^*_G(\phi,a)=\limsup \frac{1}{n}\log Z^*_n(\phi,a)$. By  \eqref{e.Z-splitting} and \eqref{eq.id},
\begin{align*}
&P^*_G(\phi,a)\leq \limsup \frac{1}{n}\log \bigl(2\max_i(Z_n^{*}(\phi|_{\Sigma^+(\mathfs G_i)},a)\bigr)\\
&=\limsup \frac{1}{n}\log \bigl(\max_i(Z_n^{*}(\phi|_{\Sigma^+(\mathfs G_i)},a)\bigr)
\leq \max_i P^*_G(\phi|_{\Sigma^+(\mathfs G_i)},a)\leq P^*_G(\phi,a).
\end{align*}
So  $P^*_G(\phi,a)=P^*_G(\phi|_{\mathfs G_i},a)$ for at least one of the two indices $i=0,1$.
Call this index $i_0$.  Then
$
P_G(\phi)\overset{(1)}{=}P_G^\ast(\phi,a)=P_G^\ast(\phi|_{\Sigma^+(\mathfs G_{i_0})},a)\overset{(2)}{\leq}
P_G(\phi|_{\Sigma^+(\mathfs G_{i_0})})\overset{(3)}{\leq} P_G(\phi)
$:
$\overset{(1)}{=}$ follows from Lemma \ref{l.p-star} and the assumption that  $\phi$ is not SPR;  and
$\overset{(2)}{\leq}$ and $\overset{(3)}{\leq}$ are due to the trivial inequalities $Z_n^\ast(\phi|_{\Sigma^+(\mathfs G_{i_0})},a)\leq
Z_n(\phi|_{\Sigma^+(\mathfs G_{i_0})},a)\leq Z_n(\phi,a)$.
So  $P_G(\phi)=P_G(\phi|_{\Sigma^+(\mathfs G_{i_0})})$ as required.

 The previous argument also shows that
$P_G^\ast(\phi|_{\Sigma^+(\mathfs G_{i_0})},a)=
P_G(\phi|_{\Sigma^+(\mathfs G_{i_0})})$. By Lemma \ref{l.p-star}, $\phi|_{\Sigma^+(\mathfs G_{i_0})}$ is not SPR.
\end{proof}

%

\begin{corollary}\label{cor.EKP-SPR}
Suppose $\sigma:\Sigma^+\to\Sigma^+$ is a topologically {transitive} TMS, and $\phi:\Sigma^+\to\R$ is a potential with summable variations, finite Gurevich pressure, and an equilibrium measure $\mu_\phi$.
\begin{enumerate}[(1)]
\item Assume that for every sequence of invariant probability measures $\mu_n$,
if $P_{\mu_n}(\phi)\to P_G(\phi)$, then $\mu_n[\un{a}]\to\mu_\phi[\un{a}]$ for all cylinders $[\un{a}]$.
Then $\phi$ must be  SPR.
\medskip
\item If $\phi$ satisfies the EKP inequality \eqref{EKP-ineq-pressure} for some $\beta>0$ and all $\psi\in\mathcal H_\beta$, then $\phi$ must be  SPR.
\end{enumerate}
\end{corollary}

\begin{proof} Clearly, (1)$\Rightarrow$(2). Therefore, it is enough to prove (1).

 It is enough to consider the topologically mixing case, because if $\Sigma^+$ has period $p$, and $\Sigma^+=\biguplus_{i=0}^{p-1}\Sigma^+_i$ is the spectral decomposition from \S\ref{s.spectral-decomposition}, then $\sigma$ satisfies
\begin{equation}\label{condition}
P_{\mu_n}(\phi)\to P_G(\phi)\Rightarrow \mu_n[\un{a}]\to\mu_\phi[\un{a}]
\end{equation}
if and only if  the topologically {\em mixing} $\sigma^p:\Sigma_0^+\to\Sigma_0^+$ satisfies \eqref{condition}. See the proof of Theorem \ref{t.optimal-EKP}.

Let $\mathfs G$ denote the graph associated to $\Sigma^+$. Assume without loss of generality that $\mathfs G$ has at most one edge $a\to b$ for every ordered pair of vertices $(a,b)$.

By assumption, $\phi$ has an equilibrium measure, whence by \cite{Buzzi-Sarig}, $\phi$ is positively recurrent. Suppose by way of contradiction that $\phi$ is not SPR. By the previous theorem, there is a proper subgraph $\mathfs G'\subset\mathfs G$ such that $\sigma:\Sigma^+(\mathfs G')\to \Sigma^+(\mathfs G')$ is topologically mixing, and
$$
P_G(\phi|_{\Sigma^+(\mathfs G')})=P_G(\phi).
$$
By the variational principle, there are invariant probability measures $\mu_n$ on $\Sigma^+(\mathfs G')$ such that
$
P_{\mu_n}(\phi|_{\Sigma^+(\mathfs G')})\xrightarrow[n\to\infty]{}P_G(\phi|_{\Sigma^+(\mathfs G')})
$.
 Since $P_{\mu_n}(\phi|_{\Sigma^+(\mathfs G')})=P_{\mu_n}(\phi)$ and $
P_G(\phi|_{\Sigma^+(\mathfs G')})=P_G(\phi)
$,
$P_{\mu_n}(\phi) \xrightarrow[n\to\infty]{}P_G(\phi)$.

By \eqref{condition},
$
|\mu_n[\un{a}]-\mu_\phi[\un{a}]|\xrightarrow[n\to\infty]{}0
$
for all cylinders $[\un{a}]$ in $\Sigma^+$.
But this is false for any cylinder of the form $[a,b]$ where $a\to b$ is an edge which appears in $\mathfs G$ but not in $\mathfs G'$: For this cylinder $\mu_n[a,b]=0$ for all $n$, because $\mu_n$ are supported in $\Sigma^+(\mathfs G')$. But $\mu_\phi[a,b]\neq 0$ because equilibrium measures of potentials with summable variations on topologically mixing TMS always have full support, see \cite{Buzzi-Sarig}.
\end{proof}

The following generalizes a result of Salama \cite[Theorem 2.3]{salama1988topological}, whose proof has been corrected in \cite[Theorem 2.7]{Ruette} and in  \cite{fiebig1996symbolic}.
The case of  Markovian potentials ($\phi$ such that $\var_2(\phi)=0$) was done in \cite[Theorem 3.15]{Gurevich-Savchenko}.
\begin{corollary}\label{cor:salama}
Let $\sigma:\Sigma^+\to\Sigma^+$ be a topologically mixing TMS with finite Gurevich entropy.
Suppose $\phi:\Sigma^+\to\R$ is bounded from above, has summable variations, and has finite Gurevich pressure.
Then $\phi$ is SPR if and only if $P_G(\phi|_{\Sigma^+(\mathfs G')})<P_G(\phi|_{\Sigma^+(\mathfs G)})$ for every subgraph $\mathfs G'$ of $\mathfs G$ for which $\sigma:\Sigma^+(\mathfs G')\to \Sigma^+(\mathfs G')$ is topologically mixing, and $\Sigma^+(\mathfs G')\subsetneq\Sigma^+(\mathfs G)$.
\end{corollary}
\noindent
{\em Remark.\/} The assumption that the Gurevich entropy is finite  is only needed for the \textit{if} direction.
\begin{proof}
Theorem~\ref{prop:ruette_generalized} implies the $\Leftarrow$ implication by contraposition:
If $\phi$ were not SPR, then  Theorem~\ref{prop:ruette_generalized} would have provided a proper subgraph $\mathfs G'$ such that $\sigma:\Sigma^+(\mathfs G')\to \Sigma^+(\mathfs G')$ is topologically mixing, and so that $P_G(\phi|_{\Sigma^+(\mathfs G')})=P_G(\phi)$.

For the other direction, assume by  contradiction that $\phi$ is SPR, but that there is a proper subgraph $\mathfs G'$ such that $\sigma:\Sigma^+(\mathfs G')\to \Sigma^+(\mathfs G')$ is topologically mixing and
$P_G(\phi|_{\Sigma^+(\mathfs G')})=P_G(\phi)$.

Fix some vertex $a$ of $\mathfs G'$. By Lemma \ref{l.p-star}, $P_G^\ast(\phi,a)<P_G(\phi)$, whence
$$
P_G^\ast(\phi|_{\Sigma^+(\mathfs G')},a)\leq P_G^\ast(\phi,a)<P_G(\phi)=P_G(\phi|_{\Sigma^+(\mathfs G')}),
$$
whence  $\phi|_{\Sigma^+(\mathfs G')}$ is SPR.
In particular, $\phi|_{\Sigma^+(\mathfs G')}$ is
 positively recurrent.  Let $m$ denote the  RPF measure of $\phi|_{\Sigma^+(\mathfs G')}$. Since $\Sigma^+$ has finite Gurevich entropy, $m$ has finite entropy. By Theorem \ref{t.eq-measure}, $m$  is  an equilibrium measure of $\phi|_{\Sigma^+(\mathfs G')}$.

Now let $m'$ denote the shift invariant probability measure
on $\Sigma^+(\mathfs G)$, given by $m'(E):=m(E\cap\Sigma^+(\mathfs G'))$. Clearly
\[P_{m'}(\phi)=P_{m}(\phi|_{\Sigma^+(\mathfs G')})=P_G(\phi|_{\Sigma^+(\mathfs G')})=P_G(\phi).\]
So $m'$ defines an equilibrium measure on $\Sigma^+(\mathfs G)$. By Theorem~\ref{t.eq-measure}, this is the unique equilibrium measure. But now we have a contradiction, because $\mathrm{supp}\, m'=\Sigma^+(\mathfs G')\subsetneq \Sigma^+(\mathfs G)$, whereas the equilibrium measure of a potentials with summable variations on $\Sigma(\mathfs G)$  are globally supported, because they are RPF measures.
\end{proof}

The scenario when the EKP inequality \eqref{EKP-ineq-pressure} holds, can also be characterized in terms of the following object,
called the {\em pressure at infinity} of $\phi$:
$$
P_\infty(\phi):=\sup\left\{\limsup_{n\to\infty}P_{\mu_n}(\phi)\right\}
\begin{array}{l}
\text{where the supremum is over all sequences of $\mu_n$ in}\\
\text{$\mathcal M_\phi(\Sigma^+)$ such that $\mu_n[\un{a}]\to 0\text{ for all cylinders}$}.
\end{array}
$$
$P_\infty(0)$  is called the {\em entropy at infinity}, see \cite{Buzzi-Ruette}, \cite{iommi2019escape}; For the pressure at infinity in a different setup, see \cite{gouezel2020pressure}.
\begin{theorem}
  Let $\Sigma^+$ be a  topologically transitive countable Markov shift which
has finite Gurevich entropy.
Let $\phi$ be a $\theta$-weakly H\"older continuous potential such that $\sup\phi<\infty$.
The following are equivalent:
\begin{enumerate}[(1)]
\item $\phi$ is strongly positively recurrent;
\item $\phi$ satisfies the EKP inequality \eqref{EKP-ineq-pressure} for all $\psi\in\mathcal H_\beta$, with $e^{-\beta}\leq \theta$;
\item $P_\infty(\phi)<P_G(\phi)$.
\end{enumerate}
\end{theorem}
\noindent
(1)$\Rightarrow$(2) is Theorem
\ref{t.suboptimal-EKP}; (2)$\Rightarrow$(3) is trivial; and (3)$\Rightarrow$(1) uses the following lemma,  due to Ruette  in the special case $\phi\equiv 0$  \cite[Prop 3.4.2]{ruette2001chaos}, \cite[Cor. 3.4]{Ruette}.

The following lemma holds under a general condition,  introduced in \cite{iommi2021space}, and called the {\em $\mathcal{F}$-property}. A TMS $\Sigma^+$ has the $\mathcal{F}$-property, if for each state $a$ and every positive integer $n$, $Z_N(0,a)<\infty$. Every locally compact TMS, and every topologically transitive TMS with finite Gurevich entropy has the $\mathcal{F}$-property:
In the locally compact case,
$Z_N(0,a)<\infty$ because the outgoing degree of every vertex is finite,
and in the finite entropy case, $Z_N(0,a)<\infty$ by  \eqref{Z-n-finite}.

In \cite[Lemma 4.16]{iommi2021space} it is proven that any (non-compact) countable TMS with the $\mathcal{F}$-property has probability measures that escape to infinity, in the sense that the mass of every cylinder tends to zero. The next lemma says that  if the SPR property fails, these measures can be chosen to be asymptotically equilibrium measures.

\begin{lemma}\label{c.SPR-and-escape-to-infty}
Suppose $\Sigma^+$ is a topologically transitive countable Markov shift with the $\mathcal{F}$-property.
Let $\phi$ be a function with summable variations and finite Gurevich pressure. If $\phi$ is not SPR,
  then there exists a sequence of
$\sigma$-invariant ergodic probability measures $\mu_n$ such that
\begin{equation}\label{escape-to-infinity}
P_{\mu_n}(\phi)\xrightarrow[n\to\infty]{} P_G(\phi)\text{,  but }\mu_n[\un{a}]\xrightarrow[n\to\infty]{} 0\text{ for {\em all} cylinders $[\un{a}]$}.
\end{equation}
\end{lemma}

\begin{proof}

We will say that a subgraph $\mathfs G'\subset\mathfs G$ is {\em full}, if (a) $\mathfs G'$ is strongly connected, (b) $P_G(\phi|_{\Sigma^+(\mathfs G')})=P_G(\phi)$, and (c) $\phi|_{\Sigma^+(\mathfs G')}$ is not SPR.

 Choose an enumeration $V=\{a_1,a_2,\ldots,\}$ of the set of vertices of $\mathfs G$. We will construct full subgraphs $\mathfs H_N\subset \mathfs G$ such that for each $1\leq i\leq N$,  either $a_i$ is not a vertex of $\mathfs H_N$, or every path in $\mathfs H_N$ from $a_i$ to $a_i$ has length at least $N$. Then the  upper density of $a_i$ in infinite $\mathfs H_N$-paths is at most $1/N$ $(1\leq i\leq N)$.
By the ergodic theorem, every invariant probability measure on $\Sigma^+(\mathfs H_N)$  must  give  $[a_1],\ldots,[a_N]$ measure $\leq \frac{1}{N}$. By the variational principle on $\Sigma^+(\mathfs H_N)$, there are ergodic  invariant measures  $\mu_N'$ on $\Sigma^+(\mathfs H_N)$  such that
$$
|P_{\mu_N'}(\phi)-P_G(\phi)|<\frac{1}{N}\ , \ \mu_N'[a_i]\leq\frac{1}{N}\ \ \ \ \ (i=1,\ldots,N).
$$
Then the measures  $\mu_N(E)=\mu_N'(E\cap \Sigma^+(\mathfs H_N))$ on $\Sigma^+$ satisfy \eqref{escape-to-infinity}.


A {\em loop of length $N$} is an admissible word of the form $(a,\xi_1,\ldots,\xi_{N-1},a)$. An  $\mathfs H$-loop is a loop inside the  subgraph $\mathfs H$. Note that it is not required to be a first return loop, i.e.\ $\xi_i$ can be equal to $a$.
Let
\[
W_N(\mathfs H,a):=\{(a,b):(a,b)\text{ can be extended to an admissible $\mathfs H$-loop of length $N$}\}.
\]
More generally, for a cylinder $[\un a] = [a,x_1,\dots,x_{n-1}]$ of length $n\le N$, define
\[
  W_N(\mathfs H,\un a ):=\{ (\un a, b): (\un a,b)\text{ can be extended to an admissible $\mathfs H$-loop of length $N$}\}
\]
where $(\un a, b) = (a,x_1,\dots,x_{n-1}, b)$.
Clearly $|W_N(\mathfs H,\un a)|, |W_{N}(\mathfs H,a)|\leq Z_N(0,a)$, therefore
$|W_N(\mathfs H,\un a)|, |W_{N}(\mathfs H,a)|<\infty$  by the $\mathcal{F}$-property.

Fix $a:=a_1$. We prove the following statement:

\medskip
\noindent
{\em{Claim:\/}}
For any full subgraph $\mathfs H\subset\mathfs G$, for every $N\geq 2$, there exists a full subgraph $\mathfs G_N\subset\mathfs H$ without any loops at $a$ of length $< N$.
Equivalently,
\[
  |W_\ell(\mathfs G_N, a)| =  0\text{ for all }\ell<N.
\]

\noindent
{\em Proof.\/} Suppose $\mathfs H\subset\mathfs G$ is full. We prove the claim by induction on $N$.

{\em{Beginning of the induction ($N=2$):}}
Let $\mathfs G_2\subseteq\mathfs H$ denote the graph $\mathfs H$ with the edge $a\to a$ removed.
$\mathfs G_2$ is strongly connected, and $Z_n^\ast(\phi|_{\Sigma(\mathfs G_2)},a)=Z_n^\ast(\phi,a)$ for all $n>1$.
By Lemma~\ref{l.p-star}, if $\phi$ is not SPR, then
$
P_G(\phi)=P_G^\ast(\phi,a)=P_G^\ast(\phi|_{\Sigma^+(\mathfs G_2)},a)\leq P_G(\phi|_{\Sigma^+(\mathfs G_2)})\leq P_G(\phi),
$
whence
$
P_G^\ast(\phi|_{\Sigma^+(\mathfs G_2)},a)= P_G(\phi|_{\Sigma^+(\mathfs G_2)})=P_G(\phi)
$.
So  $\mathfs G_2$ is full.
Since $a\to a$ is not an edge in $\mathfs G_2$, $|W_1(\mathfs G_2, a)|=0$.

{\em{Induction step:}}
Assume by induction that the claim  holds for $N$. Then there is a full sub-graph $\mathfs G_{N}\subset\mathfs H$ such that $|W_n(\mathfs G_{N},a)|=0$ for any $n<N$. In particular, $\mathfs G_N$ does not contain the edge $a\to a$.
If $|W_N(\mathfs G_{N},a)|=0$, we can take $\mathfs G_{N+1}:=\mathfs G_N$. It remains to treat the case $|W_N(\mathfs G_{N},a)|>0$.

We first construct a full subgraph $\overline{\mathfs G}_{N}$ of $\mathfs G_{N}$  for which $|W_N(\mathfs G_{N},a)|\leq 1$.
If $|W_N(\mathfs G_{N},a)|= 1$,
set $\overline{\mathfs G}_{N}:=\mathfs G_{N}$.
Otherwise, we split the edges  going out from $a$ into two disjoint subsets, $E_0$ and $E_1$,
each containing at least one element of  $W_N(\mathfs G_{N},a)$.
Let $\mathfs G_{N}^i$ be the graph obtained by removing edges $E_{1-i}$  and restricting to the strongly connected component containing $a$.
Arguing as in the proof of Theorem~\ref{prop:ruette_generalized},
we can show that  at least one of the $\mathfs G_{N}^i$ is full.
Call the full subgraph   $\mathfs G_{N}^{(1)}$.
By construction, $|W_N(\mathfs G_{N}^{(1)},a)|< |W_N(\mathfs G_{N},a)|$.
If $|W_N(\mathfs G_{N}^{(1)},a)|\leq 1$,
set $\overline{\mathfs G}_{N}:=\mathfs G_{N}^{(1)}$.
Otherwise repeat the procedure to obtain a full subgraph $\mathfs G_{N}^{(2)}\subset \mathfs G_{N}^{(1)}$ with $|W_N(\mathfs G_{N}^{(2)},a)|< |W_N(\mathfs G_{N}^{(1)},a)|$.
Repeating this, we eventually arrive to a full subgraph $\overline{\mathfs G}_{N}:=\mathfs G_{N}^{(m)}\subset\cdots\subset\mathfs G_{N}^{(1)}$ such that
$
|W_N(\overline{\mathfs G}_{N},a)|\leq 1.
$

If $|W_N(\overline{\mathfs G}_{N},a)|=0$, we finished the inductive step, and we  set $\mathfs G_{N+1} = \overline{\mathfs G}_{N}$.
If $|W_N(\overline{\mathfs G}_{N},a)|=1$, let $(a,b)$ denote the unique edge in $W_N(\overline{\mathfs G}_{N},a)$.
Then,
\begin{equation}\label{e.xi-1=b}
  \text{for every $\overline{\mathfs G}_{N}$-loop $(a,\xi_1,\ldots,\xi_{n-1},a)$, either $\xi_1=b$ or $n>N$}
\end{equation}
($n<N$ is not possible because $|W_n(\mathfs G_N,a)|=0$ and  $\mathfs G_N\supset\overline{\mathfs G}_{N}$.) Note that $b\neq a$, because  $|W_1(\mathfs G_N,a)|=0$.
Let $\overline{\mathfs G}_{N}'$ denote the irreducible component of $a$, after removing the edge $a\to b$.
If $\overline{\mathfs G}_{N}'$ is full, we call it $\mathfs G_{N+1}$ and we finished the induction.
If  $\overline{\mathfs G}_{N}'$ is not full, then  by the edge removal argument, the irreducible component of $a$ after removing all edges $a\to\xi$ {\em except} for $a\to b$, is full. Call it $\overline{\mathfs G}_{N}^{ab}$.
(If $b$ is the unique neighbor of $a$, $\overline{\mathfs G}_{N}^{ab}=\overline{\mathfs G}_N'$.)

We now apply the preceding procedure to $\overline{\mathfs G}_{N}^{ab}$.
That is, we split the neighbors at $b$ in two, till we obtain a full subgraph
$\widetilde{\mathfs G}_{N}^{ab}$ such that $|W_N(\widetilde{\mathfs G}_{N}^{ab},ab)|\leq 1$.
If $|W_N(\widetilde{\mathfs G}_{N}^{ab},ab)|=0$, then by \eqref{e.xi-1=b}, $|W_N(\widetilde{\mathfs G}_{N}^{ab},a)|=0$,
we let $\mathfs G_{N+1}:=\widetilde{\mathfs G}_{N}^{ab}$, and we finish the induction. Otherwise, we let $(a,b,c)$ denote the unique member in $W_N(\widetilde{\mathfs G}_{N}^{ab},ab)$. Note that $c\neq a$, because
$\widetilde{\mathfs G}_{N}^{ab}\subset \mathfs G_N$ and $|W_3(\mathfs G_N,a)|=0$.

We now rename $\mathfs G_N^{abc}:=\widetilde{\mathfs G}_{N}^{ab}$ and repeat  the previous procedure with $\mathfs G_N^{abc}$ replacing $\mathfs G_N^{ab}$ to obtain a full subgraph $\mathfs G_N^{abcd}$  with a unique member $(a,b,c,d)$ in $W_N(\mathfs G_N^{abcd},a)$. Again, $d\neq a$. We continue this way until $W_N(\cdot,a)$ is empty or until  we obtain an admissable word $\un{a}$ of length $N$ such that $W_N(\mathfs G^{\un{a}}_N,a)=\{\un{a}\}$, and so that  $a_i\neq a$ for all $i$ except the first one.

Let $z:=$last symbol in $\un{a}$. We wish to remove the edge $z\to a$ from  $\overline{\mathfs G}_{N}^{\un a}$.
To do this we apply the edge removal process, but this time to the {\em incoming} edges at $a$.
We split these into the set $\{z\to a\}$ and its complement.
The irreducible component of $a$, after removing all edges $\xi\to a$ {\em except} for $z\to a$, is just a single loop $(\un{a},a)$ so its $Z_n^\ast$ are equal to zero for all $n>N$.
It follows that the  irreducible component of $a$ after removing {\em just} $z\to a$ must be full. We call this graph $\mathfs G_{N+1}$.  $W_{N}(\mathfs G_{N+1},a)=\emptyset$ by construction, and $W_n(\mathfs G_{N+1},a)=\emptyset$ for $n<N+1$ because $\mathfs G_{N+1}\subset \mathfs G_N$,
and we finished the induction. The claim is proved.

\medskip
\noindent
{\em{Conclusion of the proof:}} The  claim we just proved gives us a full subgraph $\mathfs G_N$ so that $|W_n(\mathfs G_N,a)|=0$ for all $n<N$. In particular, the upper density of $a$ in every infinite $\mathfs G_N$-path is at most $1/N$. This takes care of $a=a_1$.

If $a_2\not\in \mathfs G_N$, its density in $\mathfs G_N$-paths is zero. Otherwise we apply the claim to $\mathfs H:=\mathfs G_N$ with $a:=a_2$, and obtain  a full subgraph with $a_1$ and $a_2$ appearing with upper density at most $\frac{1}{N}$. After $N$ steps like this we arrive to the graph $\mathfs H_N$ from the beginning of the proof.
\end{proof}

\section{Two sided topological Markov shifts}

Suppose $\mathfs G$ is a countable directed graph.
The {\em two-sided} topological Markov shift associated to $\mathfs G$ is the dynamical system with the space
$$
\Sigma=\Sigma(\mathfs G):=\{\un{x}=(\ldots, x_{-1},x_0,x_1,\ldots): x_i\in S,  x_i\to x_{i+1}\text{ for all }i\},
$$
the action $\sigma(\un{x})_i=x_{i+1}$, and the metric $d(\un{x},\un{y})=\exp(-\inf\{|i|:x_i\neq y_i\})$.

The definitions we gave in \S\ref{s.setup} for one-sided shifts of the Gurevich pressure, Gurevich entropy, the spectral decomposition and SPR topological Markov shifts  extend verbatim to the two-sided case, after  replacing all the $\Sigma^+$ by $\Sigma$.

The definition of the SPR property for potentials has a similar extension to the two-sided case, except that now to define the discriminant, we need to work with the induced shift on $\ov{\Sigma}:=\ov{S}^{\mathbb Z}$, instead of $\ov{\Sigma}^+=\ov{S}^{\mathbb N\cup\{0\}}$.

With these definitions in place, Theorems~\ref{t.optimal-EKP} and \ref{t.suboptimal-EKP} on the EKP inequality for SPR potentials  extend to the two-sided setup without much difficulty.
We  explain why.

Any weakly H\"older continuous $\phi:\Sigma\to\R$ is cohomologous via a bounded  continuous transfer function to a ``one-sided" function of the form $\phi^+\circ\pi$, where $\phi^+$ is a weakly H\"older continuous function on $\Sigma^+$, and $\pi:\Sigma\to\Sigma^+$ is the natural projection, $\pi(\un{x})=(x_0,x_1,\ldots)$,
\cite{Sinai-Gibbs}, \cite{Bowen-LNM}, \cite{Daon}.

The pressure function is defined in terms of sums over periodic orbits. Such sums do not change if we change $\phi$ by a coboundary. Therefore, it is easy to see that $\phi$ is SPR if and only if $\phi^+$ is SPR, and $P_G(\phi)=P_G(\phi^+)$.

 If $\mu$ is a shift invariant measure on $\Sigma$, then $\mu^+:=\mu\circ\pi^{-1}$ is a shift invariant probability measure on $\Sigma^+$, and it is easy to see that $h_{\mu^+}(\sigma)=h_{\mu}(\sigma)$.
In addition, coboundaries with bounded continuous transfer functions are absolutely integrable with zero integral for all shift invariant probability measures, so $\int\phi d\mu=\int\phi^+ d\mu^+$. So $\mu\in\mathfs M_\phi(\Sigma)$ if and only if  $\mu^+\in\mathfs M_{\phi^+}(\Sigma^+)$, and in this case
$
 P_{\mu}(\phi)=P_{\mu^+}(\phi^+).
$

In particular, $\mu$ maximizes $ P_{\mu}(\phi)$ if and only if $\mu^+$ maximizes $P_{\mu^+}(\phi^+)$, and therefore $\mu_\phi$ is the equilibrium measure of $\phi$ on $\Sigma$, if and only if $(\mu_\phi)^+$ is the equilibrium measure of $\phi^+$ on $\Sigma^+$.
We are therefore at liberty to write
$$
\mu_\phi^+=\mu_{\phi^+}.
$$
It is now a simple matter to see that the EKP inequalities \eqref{EKP-ineq} and \eqref{EKP-ineq-pressure} for $\Sigma^+$ and $\phi^+$, imply the EKP inequalities \eqref{EKP-ineq} and \eqref{EKP-ineq-pressure} for $\Sigma$ and $\phi$, {\em  provided  the test function $\psi$ belongs to }
$
\mathcal H_\beta^+(\Sigma):=\{\psi^+\circ\pi\circ \sigma^n: \psi\in\mathcal H_\beta(\Sigma^+), n\in\Z\}.
$

Since $\mathcal H_\beta^+(\Sigma)$ is dense in the space of $\beta$-H\"older continuous functions on $\Sigma$,  \eqref{EKP-ineq} and \eqref{EKP-ineq-pressure} follow for all $\beta$-H\"older continuous functions on $\Sigma$.

\medskip
\noindent
{\em Note added in proof.\/}
J. Buzzi, S. Crovisier and O.S. have recently shown that every  topologically transitive $C^\infty$ diffeomorphism on a closed smooth surface admits a H\"older continuous symbolic coding by an SPR countable Markov shift, with finite Gurevich entropy. By the results of this paper, such diffeomorphisms satisfy the EKP inequality for all H\"older continuous functions $\psi$ on the manifold. Details will appear elsewhere.

\bibliographystyle{alpha}
\newcommand{\etalchar}[1]{$^{#1}$}
\def\cprime{$'$} \def\cprime{$'$} \def\cprime{$'$} \def\cprime{$'$}

\end{document}